\RequirePackage{fix-cm}
\documentclass[smallextended]{svjour3}       
\smartqed  
\usepackage{graphicx}
\usepackage{amsmath}
\usepackage{mathtools}
\usepackage{caption}
\usepackage{subfigure}
\usepackage{adjustbox,lipsum}
\usepackage{float}
\usepackage{bm}
\usepackage{adjustbox}
\usepackage{amsfonts}
\usepackage{wrapfig}
\usepackage{siunitx}
\usepackage{tabularx}
\usepackage{diagbox}
\usepackage{makecell}
\usepackage{multirow}
\usepackage{rotating}

\usepackage[square,numbers]{natbib}
\usepackage{latexsym}
 \journalname{pre-print}
\begin{document}

\title{Stability analysis and parameter classification of a reaction-diffusion model on non-compact circular geometries
}


\author{Wakil Sarfaraz         \and
        Anotida Madzvamuse 
}


\institute{W. Sarfaraz \at
              University of Sussex, School of Mathematical and Physical Sciences, Department of Mathematics, Pevensey 3, Brighton, BN1 9QH, UK \\
              Tel.: +44 (0)7832903102\\
              \email{wakilsarfaraz@gmail.com}           
           \and
           A. Madzvamuse \at
              University of Sussex, School of Mathematical and Physical Sciences, Department of Mathematics, Pevensey 3, Brighton, BN1 9QH, UK \\
               Tel.: +44 1273 873529\\
                \email{A.Madzvamuse@sussex.ac.uk}
}

\date{Received: date / Accepted: date}

\maketitle

\begin{abstract}
This work explores the influence of domain size of a non-compact two dimensional annular domain on the evolution of pattern formation that is modelled by an \textit{activator-depleted} reaction-diffusion system. A closed form expression is derived for the spectrum of Laplace operator on the domain satisfying a set of homogeneous conditions of Neumann type both at inner and outer boundaries. The closed form solution is numerically verified using the spectral method on polar coordinates. The bifurcation analysis of \textit{activator-depleted}  reaction-diffusion system is conducted on the admissible parameter space under the influence of two bounds on the parameter denoting the thickness of the annular region. The admissibility of Hopf and transcritical bifurcations is proven conditional on the domain size satisfying a lower bound in terms of reaction-diffusion parameters. The admissible parameter space is partitioned under the proposed conditions corresponding to each case, and in turn such conditions are numerically verified by applying a method of polynomials on a quadrilateral mesh. Finally, the full system is numerically simulated on a two dimensional annular region using the standard Galerkin finite element method to verify the influence of the analytically derived conditions on the domain size for all types of admissible bifurcations. 
\end{abstract}

\section{Introduction}\label{intro}
The application of reaction-diffusion systems (RDS) to the theory of pattern formation dates back to the work of a renowned British scientist by the name of Alan Turing, 1912-1954. Turing in his seminal paper \cite{paper1} presented a consistent account of the details and mathematical formalism showing that reaction-diffusion systems can be responsible for the emergence of pattern formation in nature. It has become an attractive area of research for scholars in applied mathematics \cite{paper2, paper3, paper4, paper5, paper6, paper7, paper8} to investigate and quantify the behaviour of a set of reaction-diffusion equations as an evolving dynamical system. Systems of reaction-diffusion equations that model the evolution of pattern formation in nature are often a set of non-linear parabolic equations \cite{paper5, paper7, paper9, paper10, paper17, paper18, paper21, paper23}, whose solution is seldom analytically retrievable. The nature and complexity of these equations make numerical approaches \cite{paper3, paper6, paper10, paper12, paper14, paper15} a necessary tool to investigate these systems \cite{paper6, paper8, paper9, paper10, paper12, paper15, paper16, paper20, paper21, paper22}. Numerical approaches in their own right provide a partial insight to obtain an empirical understanding of the spatio-temporal behaviour of the dynamics governed by RDS, since it requires a verified analysis and classification of the parameter spaces \cite{paper38, paper47} from which the values of the relevant parameters of a particular RDS are to be chosen such that these parameter values are within the bifurcation region of a particular expected behaviour in the evolving dynamics. Despite that a robust investigation of RDS in the context of pattern formation requires a computational approach to visualise the evolution of the emerging pattern \cite{book1, paper23, paper24, paper31, paper32, thesis1}, it is essential that a computational approach is conducted in light of results from stability theory \cite{paper6, paper9, paper13, paper17, paper20, paper28} on the RDS, which in turn creates the necessity to perform bifurcation analysis of the parameter spaces. Spatial pattern formation and Turing type behaviour of RDS is a relatively well-studied area \cite{paper1, paper3, paper6, paper7, paper23, paper24, paper29, paper31}, in contrast to the amount of attention given to the analysis of temporal periodicity of the dynamics governed by RDS \cite{paper50, paper51, paper52, paper53, paper54, paper55}. 
From the literature review on the topic one can easily notice that analytical approaches to study Hopf and transcritical bifurcation of RDS are often conducted on particular cases using one spatial dimension. Few examples in the literature where they derive parameter spaces in the context of bifurcation analysis is the work of Madzvamuse \textit{et al}., \cite{paper20, paper18} where they compute regions of parameter space, corresponding to diffusion-driven instability with and without cross-diffusion respectively for {\it activator-depleted} RDS. Their approach to computationally find the unstable spaces is restricted to Turing spaces. One of the significant novelty from the present work is the application of a numerical method that computationally derives not only Turing spaces, but it derives a full classification of the admissible parameter spaces. Iron \textit{et al}., in \cite{paper52} provide a detailed study on the stability analysis of Turing patterns generated by \textit{activator-depleted} reaction kinetics in one spatial dimension. Despite the presentation of rigorous and well-demonstrated proofs in \cite{paper52}, their results are restricted to spatial patterns which focus on the emergence of the number of spatial peaks relative to the eigenvalues of the one-dimensional Laplace operator. The significance of the current study in contrast to \cite{paper52} is that the spectrum of Laplace operator is employed to derive the full classification of the parameter spaces, and relating the domain size to the admissibility of different types of bifurcations in the dynamics, all of which is conducted through a transferable framework as a tool to explore general RDS. Xu and Wei studied in \cite{paper53}, the \textit{activator-depleted} reaction-diffusion model under the restriction of one spatial dimension, with main focus given to the admissibility of Hopf bifurcation. The scope and strategy in \cite{paper53} is not aimed to produce any results that relate the domain size (length or radius) to the reaction-diffusion rates. Therefore, our method and the scope of this work is robust in the sense that the domain size is explicitly related to the admissibility of Hopf and transcritical bifurcations and furthermore, the analytical results are robustly verified by numerical simulations. Yi \textit{et al} also attempted to explore the bifurcation analysis and spatio-temporal patterns of a diffusive predator-prey system in one space dimension \cite{paper54}. In \cite{paper54}, the results mainly consist of theoretical claims with incomplete numerical verifications and comparing \cite{paper54} with the current work, there is no relevance demonstrated to indicate the role of domain size on the bifurcation behaviour of the RDS. Reaction-diffusion system with \textit{activator-depleted} reaction kinetics is investigated in \cite{paper37} with time delay in one-dimensional space and it is theoretically proven again with incomplete numerical verifications that Hopf bifurcation can occur with given constraints on the parameter values of the system. Liu \textit{et al}., in \cite{paper28} attempted to find Hopf bifurcation points in parameters spaces of RDSs with {\it activator-depleted} reaction kinetics. The analytical study in \cite{paper28} proves the existence of Hopf bifurcation points under some theoretical constraints on the parametrised variables, without indicating any relevance between the domain size and reaction-diffusion rates in the context of admissibility of Hopf and/or transcritical bifurcations. Our strategy proves robust upon comparison to \cite{paper28} in the sense that we derive explicit relationships between domain size and the reaction-diffusion rates. Furthermore, we employ this relationship, to fully classify the admissible parameter space for different types of bifurcations. Furthermore, the approach in \cite{paper28} is that the existence of Hopf bifurcation points is proven through incorporating parametrised variables for the model and not through employing the actual parameters of the equations, which makes their results applicable to non-realistic possibilities (negative values) of the actual parameters of the model. This drawback is effectively accounted for in the current work through conducting the analysis on the actual two-dimensional positive real parameter space, which in addition to confirming the existence of different bifurcation regions, it offers concrete quantitative classification of the parameter space that guarantees the dynamics of RDSs to exhibit the theoretically predicted bifurcations in the dynamics. The scope of the current article serves as a leap in generalising a complete methodology presented in \cite{paper38, paper47} as a self-sufficient approach to the study of reaction-diffusion systems. The contents of \cite{paper38, paper47} are mainly concerned with two dimensional rectangular and circular but compact geometries. The substantial novelty in the current work is that the methodology is expanded by introducing techniques to analytically solve the spectrum of the diffusion operator on a non-compact geometry. These analytical findings are then employed to derive bifurcation results through which the relationship between the domain size and the reaction-diffusion rates is established.    

Hence this paper is structured as follows: In Section \ref{absenceone} we present an \textit{activator-depleted} reaction-diffusion model on cartesian coordinates, which is transformed to polar coordinates due to the geometrical nature of the domain. Section \ref{linearisation} presents the process of how the model is linearised, which entails a rigorous and detailed derivation of closed-form solution to the relevant eigenvalue problem and the application of the spectral method to depict the complex valued eigenfunctions corresponding to the Laplace operator on the domain. Section \ref{parbif} provides the bulk of our analytical findings that relates the thickness of the annular region to the reaction and diffusion parameters in the context of admissibility of different types of bifurcations. Section \ref{main} contains a brief outline of the numerical method for computing the solutions of the implicit partitioning curves that fully classify the admissible parameter spaces, which verifies the existence and/or absence of the analytically proven regions and curves of bifurcation. Section \ref{fem} provides the finite element simulations of system (\ref{polarsystem}), with the parameters chosen from the theoretically proposed regions and it is in this section that the finite element solutions are shown to exhibit the analytically predicted type of behaviour in the dynamics. Section \ref{conclusion} concludes the article with some possible directions of extension of the current work.
\section{Domain and model equations}\label{absenceone}
A reaction-diffusion system of \textit{activator-depleted} class is used to model the evolution of two chemical species $u(x,y,t)$ and $v(x,y,t)$ that react and diffuse on a non-compact two dimensional circular domain $\Omega= \{(x,y)\in \mathbb{R}^2: a^2 < x^2+y^2 < b^2\}$, which consists of annular region centred in the origin of Cartesian plane with $0<a<b$. The boundary of $\Omega$ is denoted by $\partial \Omega = \{(x,y)\in \mathbb{R}^2: x^2+y^2 = a^2\}\cup\{(x,y)\in \mathbb{R}^2: x^2+y^2 = b^2\}$. The chemical species $u$ and $v$ are assumed to diffuse independently and are coupled only through non-linear terms satisfying the well known Turing type \textit{activator-depleted} reaction kinetics. 
%

The boundary of $\Omega$ is subject to homogeneous Neumann condition, which places a restriction of zero flux \cite{paper3, paper4, paper5, paper6} on $u$ and $v$ through $\partial \Omega$. Initial conditions for the \textit{activator-depleted} model are positive bounded and continuous functions \cite{paper7, paper8, paper9, paper10} with pure spatial dependence.
%
Due to the geometrical nature of $\Omega$, it is essential to conduct the relevant study on polar coordinates. 

With this setup in mind, the RDS in its non-dimensional form on polar coordinates reads as
\begin{equation}
\begin{cases}
\begin{cases}
\frac{\partial u}{\partial t} =& \triangle_p u + \gamma f(u,v), \qquad \qquad (r,\theta)\in \Omega, \quad t>0,\\
\frac{\partial v}{\partial t} =& d \triangle_p v + \gamma  g(u,v),
\end{cases} \\
\begin{cases}
\frac{\partial u}{\partial r}\big |_{r=a} = \frac{\partial v}{\partial r}\big |_{r=a}=0,\qquad  \quad \qquad (r,\theta)\in\partial\Omega, \quad t\geq 0,\\
\frac{\partial u}{\partial r}\big |_{r=b} = \frac{\partial v}{\partial r}\big |_{r=b}=0
\end{cases}\\
u(r,\theta,0)=u_0(r,\theta), \qquad v(r,\theta,0)=v_0(r,\theta), \qquad (r,\theta) \in \Omega, \quad t=0, 
\end{cases}
\label{polarsystem}
\end{equation}
where functions $f$ and $g$ are given by $f(u,v)=\alpha-u+u^2v$ and $g(u,v)=\beta-u^2v$.
\section{Linearisation and the eigenfunctions}\label{linearisation}
Let $u_s$ and $v_s$ denote the steady state solutions of system (\ref{polarsystem}), then through a straightforward algebraic manipulation it can be verified that there exists a pair of constants \cite{book1} in the form $(u_s,v_s)=(\alpha+\beta, \frac{\beta}{(\alpha+\beta)^2})$ satisfying the steady state solutions for the nonlinear reaction terms in (\ref{polarsystem}). It is worth noting that the pair $(u_s,v_s)$ is a unique set of positive real constants satisfying $f(u_s,v_s)=g(u_s,v_s)=0$ and hence, it automatically satisfies the zero-flux boundary conditions prescribed for (\ref{polarsystem}). The standard practice of linear stability theory is applied, which requires to perturb system (\ref{polarsystem}) in the neighbourhood of the uniform steady state solution in the form $(u,v)=(u_s+\bar{u},v_s+\bar{v})$, where $\bar{u}$ and $\bar{v}$ are assumed small. System (\ref{polarsystem}) is expanded using Taylor expansion for functions of two variables up to and including the linear terms, with $u$ and $v$ replaced by their corresponding expressions in terms of $u_s$, $\bar{u}$, $v_s$ and $\bar{v}$. It leads us to write system (\ref{polarsystem}) in a matrix form as 
\begin{equation}
\frac{\partial}{\partial t}\left[\begin{array}{c}
\bar{u}  \\
\bar{v} 
\end{array}\right]=\left[\begin{array}{cc}
1&0  \\
0&d 
\end{array}\right]\left[\begin{array}{c}
\triangle_p \bar{u} \\
\triangle_p \bar{v} 
\end{array}\right]+\left[\begin{array}{cc}
\frac{\partial f}{\partial u}(u_s,v_s)& \frac{\partial f}{\partial v}(u_s,v_s)  \\
\frac{\partial g}{\partial u}(u_s,v_s)&\frac{\partial g}{\partial u}(u_s,v_s)
\end{array}\right]\left[\begin{array}{c}
\bar{u} \\
\bar{v} 
\end{array}\right].
\label{vect}
\end{equation}
The next step is to derive the spectrum of the laplace operator which is a vast area of study in pure and applied mathematics \cite{paper45, book9, book10, book11} depending on the area of its application. In particular the study of the eigenvalues and eigenfunctions of the laplace operator on spherical geometry is a well-explored area of mathematics with generalised abstractions to higher dimensional spaces \cite{book9}. However, for the purpose of the current study a theoretical solution of the laplace operator on a general spherical geometry offers impractical contribution in the sense that the majority of the existing solutions \cite{book9, book10, book11} to such eigen-value problems are derived either on boundary-free and/or compact manifolds. It is due to the non-compact nature of the prescribed domain $\Omega$ being a two-dimensional annular region and the analytical application of the zero-flux boundary conditions that creates the necessity for a step-by-step derivation of the spectral solutions for the corresponding eigenvalue problem with an explicit treatment to find the particular solution that satisfies the zero-flux boundary conditions prescribed in (\ref{polarsystem}). Despite that the derivation of the general solution to the eigenvalue problem (\ref{eigen}) is presented with brevity in the body of the article, the interested reader is referred to \cite{paper47, book10, book11} for further details on the derivation. The eigenvalue problem in polar coordinates relevant to the current scenario has the form
\begin{equation}    \label{eigen}
\begin{cases}
\triangle_p w = -\eta^2 w, \qquad \eta \in \mathbb{R}, \qquad \quad (r,\theta)\in \Omega,
\\
\frac{\partial w}{\partial r}\big |_{r=a} = \frac{\partial w}{\partial r}\big |_{r=b}=0,\qquad  \quad \qquad a,b\in\mathbb{R}_+\backslash\{0\} \quad \text{ and } a<b,
\end{cases}
\end{equation}
where $\triangle_p=\frac{1}{r}\frac{\partial}{\partial r}\Big(r\frac{\partial}{\partial r}\Big)+\frac{1}{r^2}\frac{\partial^2}{\partial \theta^2}$ and $\Omega$ is the same as prescribed for system (\ref{polarsystem}). 
Employing the method of separable solution \cite{paper47, book10, book11} leads us to write the general solution of problem (\ref{eigen}) as the product of a pure radial function $R(r)$ and a phase factor $\Theta(\theta)$ in the form $w(r,\theta)=R(r)\Theta(\theta)$. Through the application of Frobenius method \cite{paper47, book10, book11} it can be shown that $R(r)$ is in fact the sum of two linearly independent Bessel's series of the first kind in the form $R(r)=R_1(r)+R_2(r)$. The application of Frobenius method creates the necessity to solve for the radial function using a linear transformation for the variable $r$ in the form $x = \eta r$, with $\eta\in\mathbb{R}$ denoting the spectral constant of proportionality. Employing such a method leads us to write the general solution of problem (\ref{eigen}) in the form 
\begin{equation}
w(r,\theta)=\big[R_1(x(r))+R_2(x(r))]\Theta(\theta),
\label{gensol}
\end{equation}
where $R_{1,2}(x(r))=\sum_{j=0}^{\infty}\frac{(-1)^jC_0x(r)^{2j\pm l}}{4^j\times j!\times(\pm l+j)\times(\pm l+j-1)\times \cdot\cdot\cdot\times(\pm l+1)}$ and $\Theta(\theta)=\exp{(il\theta)}$. In order to obtain the particular set of solutions satisfying the eigenvalue problem (\ref{eigen}), it is necessary to impose the prescribed homogenous Neumann boundary conditions. The outward flux through $\partial \Omega$ is independent of the variable $\theta$, therefore, $w(r,\theta)$ is required to satisfy the zero flux boundary conditions of the form $\frac{\partial w}{\partial r}\big |_{r=a}=0$ and $\frac{\partial w}{\partial r}\big|_{r=b}=0$ or equivalently $\frac{d R}{d r}\big|_{r=a}=\frac{d R}{d r}\big|_{r=b}=0$ must hold. It is important to realise that $R$ now implicitly depends on the variable $r$ through the relation $x=\eta r$, therefore, application of chain rule yields
\begin{equation}
\frac{d R}{d r}\big |_{r=a}= \frac{d x}{d r}\frac{d R}{d x}\big |_{x=\eta a}= \eta \frac{d R}{d x}\big |_{x = \eta a}=0,
\label{one}
\end{equation}
\begin{equation}
\frac{d R}{d r}\big |_{r=b}= \frac{d x}{d r}\frac{d R}{d x}\big |_{x=\eta b}= \eta \frac{d R}{d x}\big |_{x = \eta b}=0.
\label{two}
\end{equation}
Adding equations (\ref{one}) and (\ref{two}) we note that $R$ is required to simultaneously satisfy the equation
\begin{equation}
\eta \Big( \frac{d R}{d x}\big |_{x=\eta a}+\frac{d R}{d x}\big |_{x=\eta b}\Big)=0.
\label{three}
\end{equation}
Using the linear property of differentiation we note that $\frac{d R}{d x}\big |_{x=\eta a}=\frac{d R_1}{d x}\big |_{x=\eta a}+\frac{d R_2}{d x}\big |_{x=\eta a}$, where upon cancellation of $\eta$, given that $\eta$ is non-zero, from equation (\ref{three}) we obtain
\begin{equation}
\frac{d R_1}{d x}\big |_{x=\eta a}+\frac{d R_2}{d x}\big |_{x=\eta a}+\frac{d R_1}{d x}\big |_{x=\eta b}+\frac{d R_2}{d x}\big |_{x=\eta b}=0.
\label{four}
\end{equation}
Differentiating with respect to $x$ the infinite series expressing $R_1$ and $R_2$ and evaluating each summation at $x=\eta a$ and $x = \eta b$ respectively, we find the following equations, which are presented completely independent of the variable $x$ in the form 
\begin{equation}
\begin{split}
\sum_{j=0}^\infty u_j(l+2j)(\eta a)^{l+2j-1}+\sum_{j=0}^\infty v_j(-l+2j)(\eta a)^{-l+2j-1}\\+\sum_{j=0}^\infty u_j(l+2j)(\eta b)^{l+2j-1}+&\sum_{j=0}^\infty v_j(-l+2j)(\eta b)^{-l+2j-1}=0,
\end{split}
\label{five}
\end{equation}
where $u_j$ and $v_j$ respectively express the $jth$ coefficient of Bessel's series solutions for $R_1$ and $R_2$. These are given by 
$u_j = \frac{(-1)^jC_0}{4^j\times j!\times(l+j)\times(l+j-1)\times \cdot \cdot \cdot\times(l+1)}$ and $v_j = \frac{(-1)^jC_0}{4^j\times j!\times(-l+j)\times(-l+j-1)\times \cdot \cdot \cdot\times(-l+1)}.$
Equation (\ref{five}) can be grouped into a set of two equations each of which contains two independent summations, which are written as  
\begin{equation}
\begin{split}
\sum_{j=0}^\infty u_j(l+2j)(\eta a)^{l+2j-1}+&\sum_{j=0}^\infty u_j(l+2j)(\eta b)^{l+2j-1}\\=&\sum_{j=0}^\infty u_j(l+2j)\big[(\eta a)^{l+2j-1}+(\eta b)^{l+2j-1}]=0,
\end{split}
\label{nine}
\end{equation}
\begin{equation}
\begin{split}
\sum_{j=0}^\infty v_j(-l+2j)(\eta a)^{-l+2j-1}+&\sum_{j=0}^\infty v_j(-l+2j)(\eta b)^{-l+2j-1}\\=&\sum_{j=0}^\infty v_j(-l+2j)\big[(\eta a)^{-l+2j-1}+(\eta b)^{-l+2j-1}]=0.
\end{split}
\label{six}
\end{equation}
The set of equations given by (\ref{nine}) and (\ref{six}) can only satisfy a simultaneous relation if they are independently equal to zero \cite{paper47} and the only way it can happen is through the application of telescoping argument \cite{book13} of real analysis. Applying the telescoping argument to (\ref{nine}) and (\ref{six}) and noting that each of the series carry alternating signs from term to term (due to $(-1)^j$ in the expression for $u_j$ and $v_j$), therefore, the only way equation (\ref{nine}) can be true is if the subsequent terms within the summation cancel each other pairwise. Let $F_j$ and $F_{j+1}$ denote in equation (\ref{nine}) the terms corresponding to indices $j$ and $j+1$ respectively, then (\ref{nine}) is true if and only if $F_j+F_{j+1}=0$ for all $j\in \mathbb{N}_0$. Writing the full expressions for $u_j$ and $u_{j+1}$, the terms corresponding to index $j$ and $j+1$ respectively take the form
\begin{equation}
\begin{split}
F_j=&\frac{(-1)^jC_0(l+2j)\big[(\eta a)^{l+2j-1}+(\eta b)^{l+2j-1}]}{4^j\times j!\times(l+j)\times(l+j-1)\times \cdot \cdot \cdot\times(l+1)},\\
F_{j+1}=&\frac{(-1)^{j+1}C_0(l+2j+2)\big[(\eta a)^{l+2j+1}+(\eta b)^{l+2j+1}]}{4^{j+1}\times (j+1)!\times(l+j+1)\times(l+j)\times \cdot \cdot \cdot\times(l+1)}.
\end{split}
\label{seven}
\end{equation}
Through a similar approach let $S_j$ and $S_{j+1}$ denote the subsequent terms in the second equation in (\ref{six}), which are given by 
\begin{equation}
\begin{split}
S_j=&\frac{(-1)^jC_0(-l+2j)\big[(\eta a)^{-l+2j-1}+(\eta b)^{-l+2j-1}]}{4^j\times j!\times(-l+j)\times(-l+j-1)\times \cdot \cdot \cdot\times(-l+1)},\\
S_{j+1}=&\frac{(-1)^{j+1}C_0(-l+2j+2)\big[(\eta a)^{-l+2j+1}+(\eta b)^{-l+2j+1}]}{4^{j+1}\times (j+1)!\times(-l+j+1)\times(-l+j)\times \cdot \cdot \cdot\times(-l+1)}.
\end{split}
\label{eight}
\end{equation}
We add the pairwise terms corresponding to indices $j$ and $j+1$ for equations (\ref{nine}) and (\ref{six}) and equate their respective sums to zero. Furthermore, we rearrange for $\eta^2$ the resulting equations for both cases namely $F_{j+1}+F_j=0$ and $S_{j+1}+S_j=0$ and evaluate them at few successive indices namely $j=0,...,8$, which reveals a pattern that for every pair of $(j,j+1)$ such that $k=2j, j\in \mathbb{N}_0$, there exists $\eta_{k,l}^2$ that can be written in terms of inner radius $a$, outer radius $b$, the corresponding order of the associated Bessel's equation $l$ and a positive integer $k$ as
\begin{equation}
F_{j+1}+F_j=0 \quad \implies \quad \eta_{1,k,l}^2=\frac{4(2k+1)(l+2k+1)(l+4k)}{a^{1-l}(a^{l+1}+b^{l+1})(l+4k+2)},
\label{ten}
\end{equation}
\begin{equation}
S_{j+1}+S_j=0 \quad \implies \quad \eta_{2,k,l}^2=\frac{4(2k+1)(l+2k+1)(l+4k)}{b^{1-l}(a^{l+1}+b^{l+1})(l+4k+2)}.
\label{eleven}
\end{equation}
The prescribed boundary conditions in (\ref{eigen}) require that the corresponding eigenfunctions $w(r,\theta)$ of the Laplace operator $\triangle_p$ satisfy a simultaneous relation on the zero-flux condition both through inner (circle with radius $a$) and outer boundaries (circle with radius $b$), which in turn suggests that the behaviour of $w(r,\theta)$ given by (\ref{gensol}) at the two boundaries namely $r=a$ and $r=b$ is related through combining the expressions given by (\ref{ten}) and (\ref{eleven}) and constructing from their combination the eigenvalues that correspond to those eigenmodes that exist in the form of a perfect superposition of those independently given by (\ref{ten}) and (\ref{eleven}) respectively. Such a combination can be constructed if the eigenmode that corresponds to the condition $\frac{d R}{d r}\big|_{r=a}$ is found in perfect superposition with an eigenmode that corresponds to the condition $\frac{d R}{d r}\big|_{r=b}$. We can further observe from the expressions given in (\ref{ten}) and (\ref{eleven}), that the patterns corresponding to index $k$ do not differ from one another, except that they are different in the coefficients that depend on $a$, $b$ and $l$. Noting that (\ref{ten}) and (\ref{eleven}) each has a radial dependence on $a$ and $b$, in addition each one of them satisfies the zero flux boundary conditions at the inner and outer boundaries, therefore using the linear property of differentiation one can write the infinite set of eigenvalues of the laplace operator $\triangle_p$ with the prescribed boundary conditions given in (\ref{eigen}) as
\begin{equation}
\eta^2_{k,l}=\frac{4(a^lb+ab^l)(2k+1)(l+2k+1)(l+4k)}{ab(a^{l+1}+b^{l+1})(l+4k+2)},
\label{eigenval}
\end{equation}
where $\eta^2_{k,l}$ in (\ref{eigenval}) is constructed from a perfect superposition of the eigenmodes corresponding to eigenvalues $\eta_{1,k,l}^2$ and $\eta_{2,k,l}^2$. At the points of superposition, due to the linearity of the operator $\triangle_p$, the set of eigenfunctions (\ref{eigenval}) is obtained from the arithmetic sum in the form $\eta_{k,l}^2=\eta^2_{1,k,l}+\eta^2_{2,k,l}$. The summary of these findings is presented in the following theorem.
\begin{theorem}\label{theorem1}
	Let $w(r,\theta)$ satisfy the eigenvalue problem  on a non-compact domain $\Omega$ defined in (\ref{eigen}). Given that the associated order of Bessel's equation is chosen such that $l\in\mathbb{R}\backslash \frac{1}{2}\mathbb{Z}$, then  the full set of eigenfunctions for the laplace operator $\triangle_p$, satisfying the corresponding homogeneous Neumann boundary conditions are given by
	\begin{equation}
	w_{k,l}(r,\theta) = \big[R_1(r)+R_2(r)\big]_{k,l}\Theta_l(\theta),
	\label{eig}
	\end{equation}
	where $R_1(r)$, $R_2(r)$ and $\Theta(\theta)$ are explicitly expressed by 
	\[
	[R_1]_{k,l}(r)=\sum_{j=0}^{\infty}\frac{(-1)^jC_0(\eta_{k,l} r)^{2j+l}}{4^j j!(l+j)(l+j-1) \cdot\cdot\cdot(l+1)},
	\]
	\[
	[R_2]_{k,l}(r)=\sum_{j=0}^{\infty}\frac{(-1)^jC_0(\eta_{k,l} r)^{2j-l}}{4^j j!(-l+j)(-l+j-1) \cdot\cdot\cdot(-l+1)}
	\]
	and $\Theta_l(\theta)=\exp{(il\theta)}$.
	Furthermore, for every $w_{k,l}(r,\theta)$ corresponding to integer $k$ and the associated order of Bessel's equation $l$, there exists a real non-negative eigenvalue $\eta_{k,l}$ satisfying 
	\begin{equation}
	\eta^2_{k,l}=\frac{4(a^lb+ab^l)(2k+1)(l+2k+1)(l+4k)}{ab(a^{l+1}+b^{l+1})(l+4k+2)}.
	\label{eigenvalue1}
	\end{equation}
\end{theorem}
	\begin{proof}
		The proof consists of all the steps from (\ref{gensol}) to (\ref{eigenval}).$\square$
	\end{proof}
\subsection{Numerical experiments using spectral methods}\label{numerical1}
The spectral method \cite{book8} is applied to validate Theorem \ref{theorem1}. Let $R_{k,l}(r)=[R_1]_{k,l}(r)+[R_2]_{k,l}(r)$ then the full set of eigenfunctions given by (\ref{eig}) can be written as $w(r,\theta)=\sum_{k=0}^\infty R_{k,l}(r)\Theta_l(\theta)$. For the numerical simulation a polar mesh is generated through a combination of non-uniform chebyshev discritisation \cite{book8, paper33} in the direction of radial variable $r$ and uniform Fourier discretisation on the periodic variable $\theta$. The domain is considered annular region centered at $(r,\theta)=(0,0)$ with parameters $a=\frac{1}{2}$ and $b=1$. A spectral mesh in polar coordinates is constructed on the region $\Omega=\{(r,\theta)\in\mathbb{R}^2:\frac{1}{2}<r<1,\theta\in[0,2\pi)\}$, where a periodic Fourier grid is used to obtain a uniform angular mesh of step size $\frac{2\pi}{M}$, where $M$ is an even positive integer of the form $M=2n, n\in\mathbb{N}$. The $jth$ mesh point on the angular axis is obtained through $\theta_j=\frac{2\pi j}{M}$ for every index $j=0,...,M$. The non-uniform mesh on the radial variable $r$ is obtained by using the chebyshev discretisation formula $r_j = \cos(\frac{2\pi j}{N})$ on the interval $r\in(\frac{1}{2},1)$, where $N$ a positive integer and $j=0,...,N$. 
Figure \ref{figmesh} (a) shows a coarse structure of a combination of a uniform Fourier grid applied to the angular variable $\theta$ with $M=30$, which makes the angular step-size of $12\si{\degree}$ and a non-uniform chebyshev grid applied to $r\in(\frac{1}{2},1)$ with $N=25$. 
Figure \ref{figmesh} (b) is constructed in similar way with $M=90$ resulting in angular step-size $4\si{\degree}$ and $N=95$, which is used to depict few of the eigenmodes proposed by Theorem \ref{theorem1} with their respective approximation of the eigenvalues $\eta_{k,l}$ proposed by formula (\ref{eigenvalue1}).
The eigenmodes $w_{k,l}(r,\theta)$ given by (\ref{eig}) corresponding to $k=1,...,12$ are visualised using HSV (Hue, Saturation Value) colour encoded scheme, which is a method presented in \cite{paper34, paper35, paper36} specifically for depicting functions of complex output. Direct methods of plotting functions of two variables do not provide a meaningful representation of the formula (\ref{eig}).
It can be noted that only the angular part in the formula (\ref{eig}), namely $\Theta(\theta)$ contains imaginary parts, therefore, the variable $\theta$ is colour encoded through the application of HSV scheme and the resulting output is depicted directly on $\Omega$.  For full details on depicting complex valued functions, the interested reader is referred to \cite{paper34, paper35}. The eigenvalues corresponding to each index $k$ for a fixed value of $l\approx0.3 \approx \frac{2\pi}{20}$ are computed and presented in the respective captions in Figure \ref{eigenfunction}. The values of $\eta_{k,l}$ are also computed for combinations of positive integer $k=1,...,12$ with a variety of values for the associated order of Bessel's equation $l=\frac{2\pi}{N}, N\in\mathbb{N}$. Table \ref{Table1} shows the computed values of $\eta_{k,l}$ for different combinations of $k$ and $l$, which offers an insight into the variations of the semi-discrete spectrum of the diffusion operator $\triangle_p$, with respect to $k$ for a fixed choice of $l$ and how $\eta_{k,l}$ varies with respect to $l$ for a fixed choice of $k$. In order to obtain a pictorial representation of the variation of $\eta_{k,l}$ with respect to both $k$ and $l$ to observe how this variation is influenced by the thickness $\rho=b-a$ of the domain size $\Omega$, a finitely truncated spectral matrix that corresponds to negative and positive values of $l$ is simulated and presented in Figure \ref{varmesh}. 
Figure \ref{varmesh} (c) in particular is simulated for the choice of $\rho=b-a$ with $a=\frac{1}{2}$ and $b=1$ to encapsulate the eigenvalues that are associated to eigenmodes shown in Figure \ref{eigenfunction}. In particular, since the value of $l$ in Figure \ref{eigenfunction} is kept fixed at $l\approx0.3$, therefore, the corresponding eigenvalues can be captured from the intersection of a vertical line at $l\approx0.3$ and the spectral lines for $k=1,...,12$ in Figure \ref{varmesh} (c). The intersection points extract the eigenvalues given on the first column of those given in Table \ref{Table1}, which are precisely the values presented in each of the sub-captions in Figure \ref{eigenfunction}.
\begin{figure}[ht]
	\centering
	\small
	\subfigure[Polar grid with $N=25$ and using the\newline periodic Fourier grid with $M=30$, leadi-\newline ng to an angular step-size of $12\si{\degree}$.]{\includegraphics[width=0.49\textwidth]{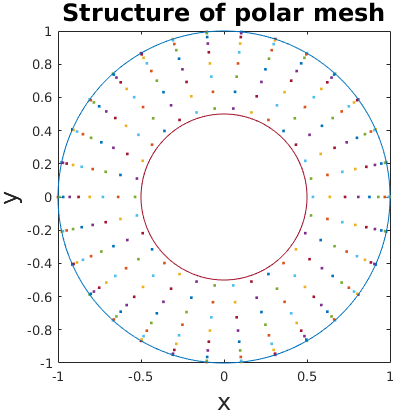}}
	\subfigure[Fully refined polar mesh with $N=95$ chebyshev grid-points and $M=90$ Fourier periodic grid-points, with angular step-size of $4\si{\degree}$.]{\includegraphics[width=0.49\textwidth]{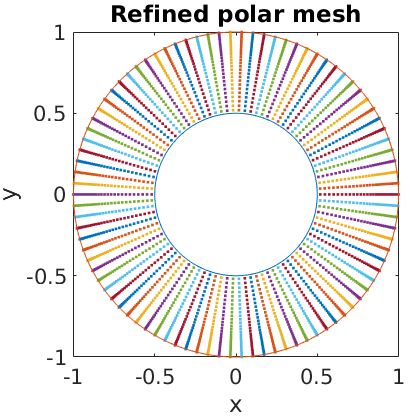}}
	\caption{Mesh generation on polar coordinates obtained from the combination of the chebyshev grid on radial axis and the Fourier periodic grid on angular axis.}
	\label{figmesh}
\end{figure}
\begin{figure}
	\begin{center}
		\subfigure[$\eta_{1,l}=7.1027$]{\includegraphics[width=0.32\textwidth]{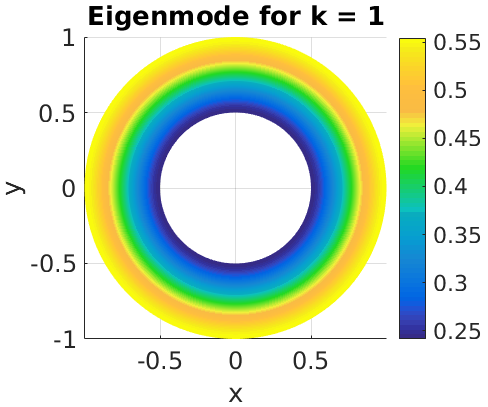}}
		\subfigure[$\eta_{2,l}=12.6266$]{\includegraphics[width=0.32\textwidth]{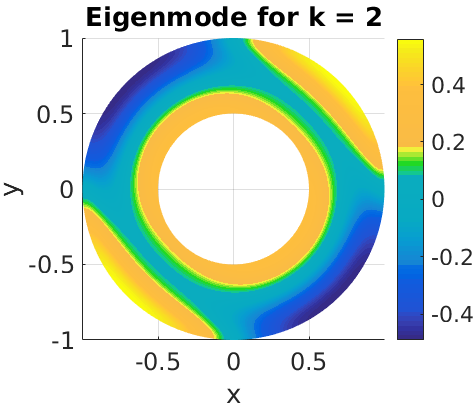}}
		\subfigure[$\eta_{3,l}=18.1149$]{\includegraphics[width=0.32\textwidth]{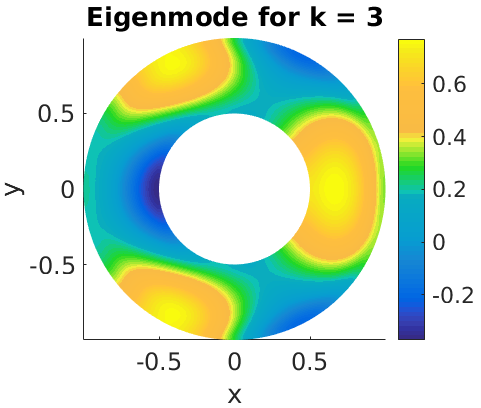}}
		\subfigure[$\eta_{4,l}= 23.5924$]{\includegraphics[width=0.32\textwidth]{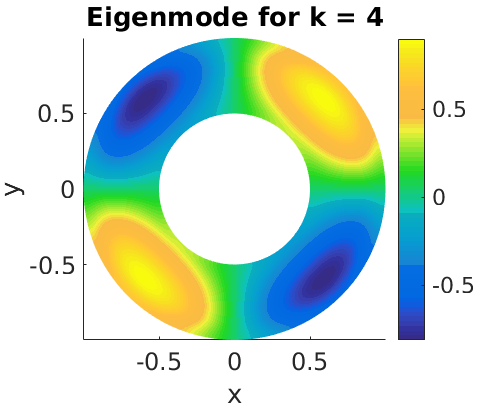}}
		\subfigure[$\eta_{5,l}=29.0652$]{\includegraphics[width=0.32\textwidth]{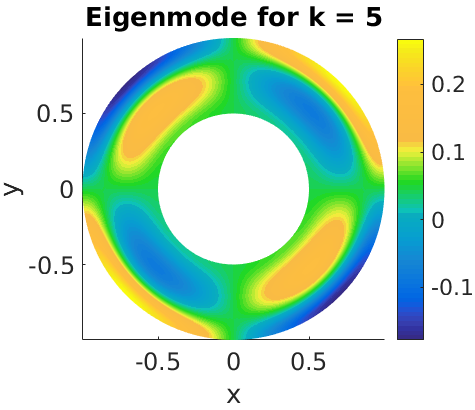}}
		\subfigure[$\eta_{6,l}=34.5354$]{\includegraphics[width=0.32\textwidth]{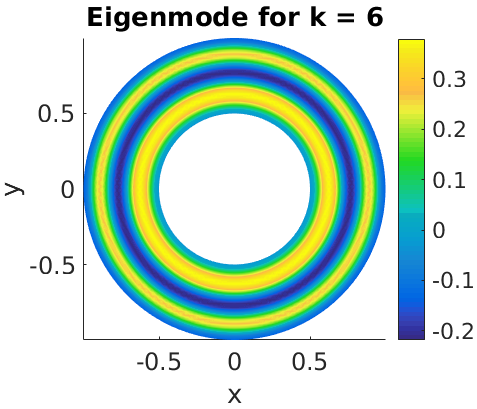}}
		\subfigure[$\eta_{7,l}=40.0041$]{\includegraphics[width=0.32\textwidth]{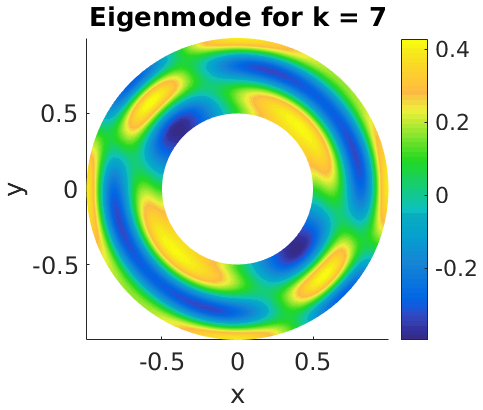}}
		\subfigure[$\eta_{8,l}=45.4719$]{\includegraphics[width=0.32\textwidth]{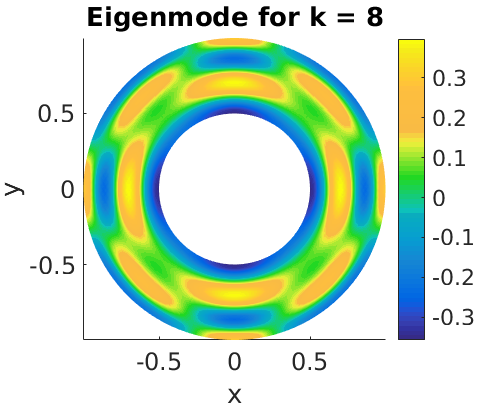}}
		\subfigure[$\eta_{9,l}=50.9391$]{\includegraphics[width=0.32\textwidth]{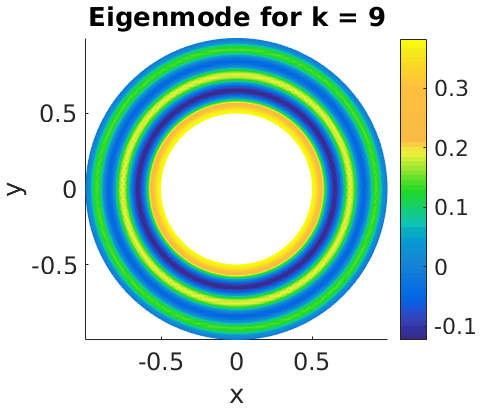}}
		\subfigure[$\eta_{10,l}=56.4057$]{\includegraphics[width=0.32\textwidth]{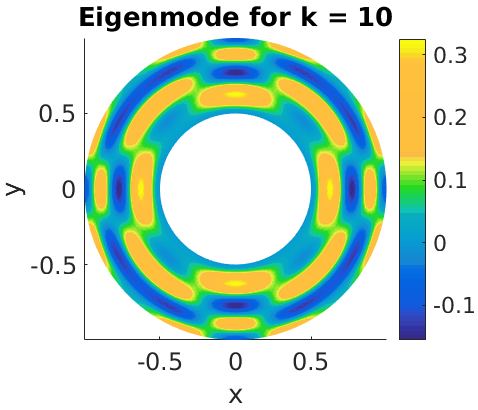}}
		\subfigure[$\eta_{11,l}= 61.8721$]{\includegraphics[width=0.32\textwidth]{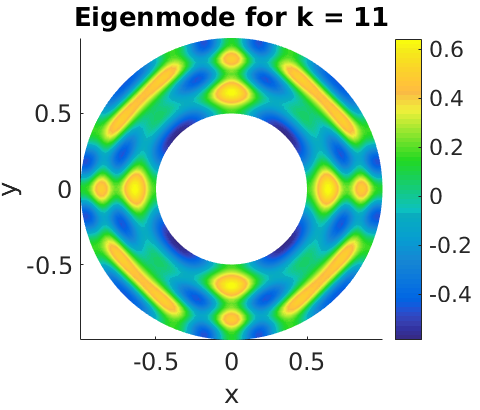}}
		\subfigure[$\eta_{12,l}=67.3382$]{\includegraphics[width=0.32\textwidth]{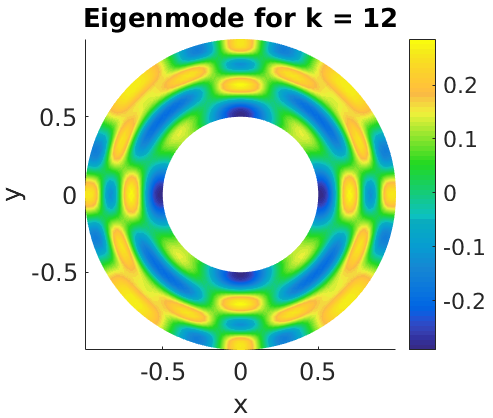}}
	\end{center}
	\caption{Colour encoded phaseplots of $w(r,\theta)$ for different values of $k$ indicated in each subcaption.}
	\label{eigenfunction}       
\end{figure}
\begin{table}[!ht]
	\caption{Numerical values of $\eta_{k,l}$ for a variety of choices of positive integer $k$ and the associated order of the Bessel's equation $l$.} 
	\centering
	\small
	\tabcolsep=0.3cm
	\noindent\adjustbox{max width=\textwidth}{
		\begin{tabular}{|c |c |c |c |c| c| c| c| c| c| c| c| c|}
			\cline{1-1}
			\hline
			\diaghead{\theadfont{\normalsize} Type of (S) }{$k$}{$l$}& $0.3$ & $1.3$& $2.3$ & $3.3$ & $4.3$ & $5.3$& $6.3$ &$7.3$ & $8.3$ &  $9.3$& $10.3$ &$11.3$\\
			\hline
			$1$&7.1027  &  7.5122  &  7.8501  &  8.2427  &  8.7082  &  9.2225  &  9.7582 &  10.2961 & 10.8254 &  11.3407 &  11.8399 &  12.3228\\
			\hline
			$2$ & 12.6266  & 12.4983  & 12.4927  & 12.7075 &  13.1098 &  13.6310 &  14.2134  & 14.8198 & 15.4290 &  16.0304 &  16.6187   &17.1920\\
			\hline
			$3$&18.1149  & 17.4447 &  17.0769  & 17.0888 &  17.3997 &  17.8974 &  18.4949 &  19.1376  & 19.7947 &  20.4503  & 21.0965  & 21.7296\\
			\hline
			$4$ &    23.5924   & 22.3758  &  21.6362   & 21.4330  &  21.6385  &  22.0976  &  22.6942  &  23.3568  &  24.0452  &  24.7384  &  25.4257  &  26.1021\\
			\hline
			$5$ &    29.0652  &  27.2996   & 26.1827   & 25.7575   & 25.8495   & 26.2611  &  26.8475  &  27.5201  &  28.2297  &  28.9502   & 29.6684   & 30.3779\\
			\hline
			$6$ &    34.5354  &  32.2191   & 30.7217  &  30.0701  &  30.0436   & 30.4021  &  30.9720   & 31.6483  &  32.3724  &  33.1135    & 33.8557  &  34.5913\\
			\hline
			$7$ &    40.0041  &  37.1361  &  35.2559  &  34.3750  &  34.2266  &  34.5281  &  35.0775  &  35.7529  &  36.4869  &  37.2437   & 38.0050   & 38.7618\\
			\hline
			$8$ &    45.4719  &  42.0513  &  39.7869   & 38.6747   & 38.4020   & 38.6438  &  39.1696  &  39.8409   & 40.5813   & 41.3503   & 42.1272   & 42.9014\\
			\hline
			$9$ &    50.9391   & 46.9654  &  44.3157   & 42.9706  &  42.5719   & 42.7519   & 43.2519  &  43.9168  &  44.6610  &  45.4396   & 46.2291  &  47.0180\\
			\hline
			$10$ &    56.4057  &  51.8786  &  48.8427   & 47.2638   & 46.7376   & 46.8544   & 47.3268   & 47.9834   & 48.7296   & 49.5155  &  50.3156  &  51.1170\\
			\hline
			$11$ &    61.8721   & 56.7912   & 53.3685  &  51.5549  &  50.9003  &  50.9526  &  51.3961   & 52.0428   & 52.7894  &  53.5811  &  54.3901   & 55.2021\\
			\hline
			$12$ &    67.3382  &  61.7032  &  57.8933   & 55.8444  &  55.0605   & 55.0474   & 55.4610   & 56.0967  &  56.8423  &  57.6385   & 58.4549  &  59.2761\\
			\hline
	\end{tabular}}
	\label{Table1}
\end{table}
\begin{figure}[H]
\begin{center}
\subfigure[]{\includegraphics[width=0.32\textwidth]{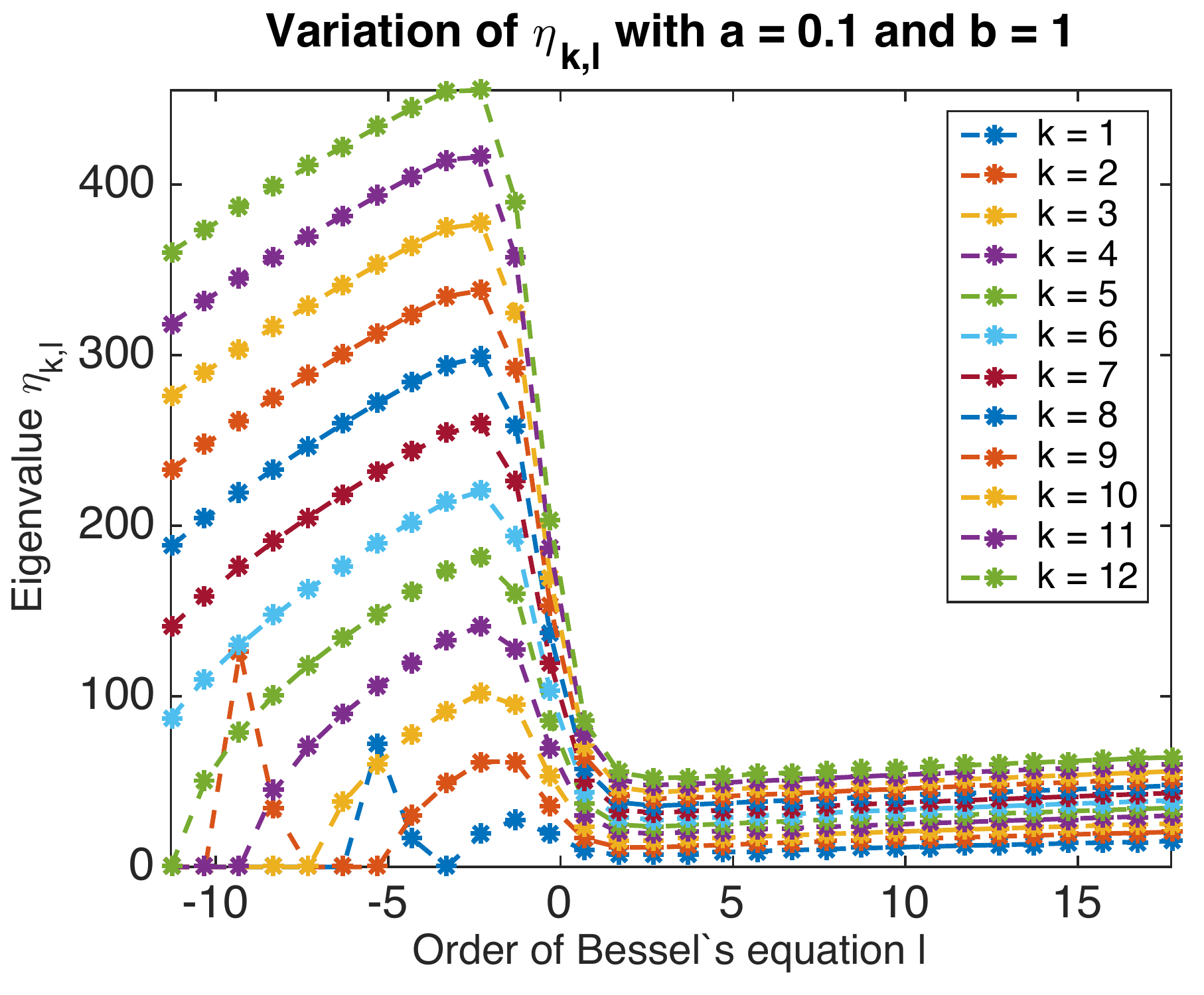}}
\subfigure[]{\includegraphics[width=0.32\textwidth]{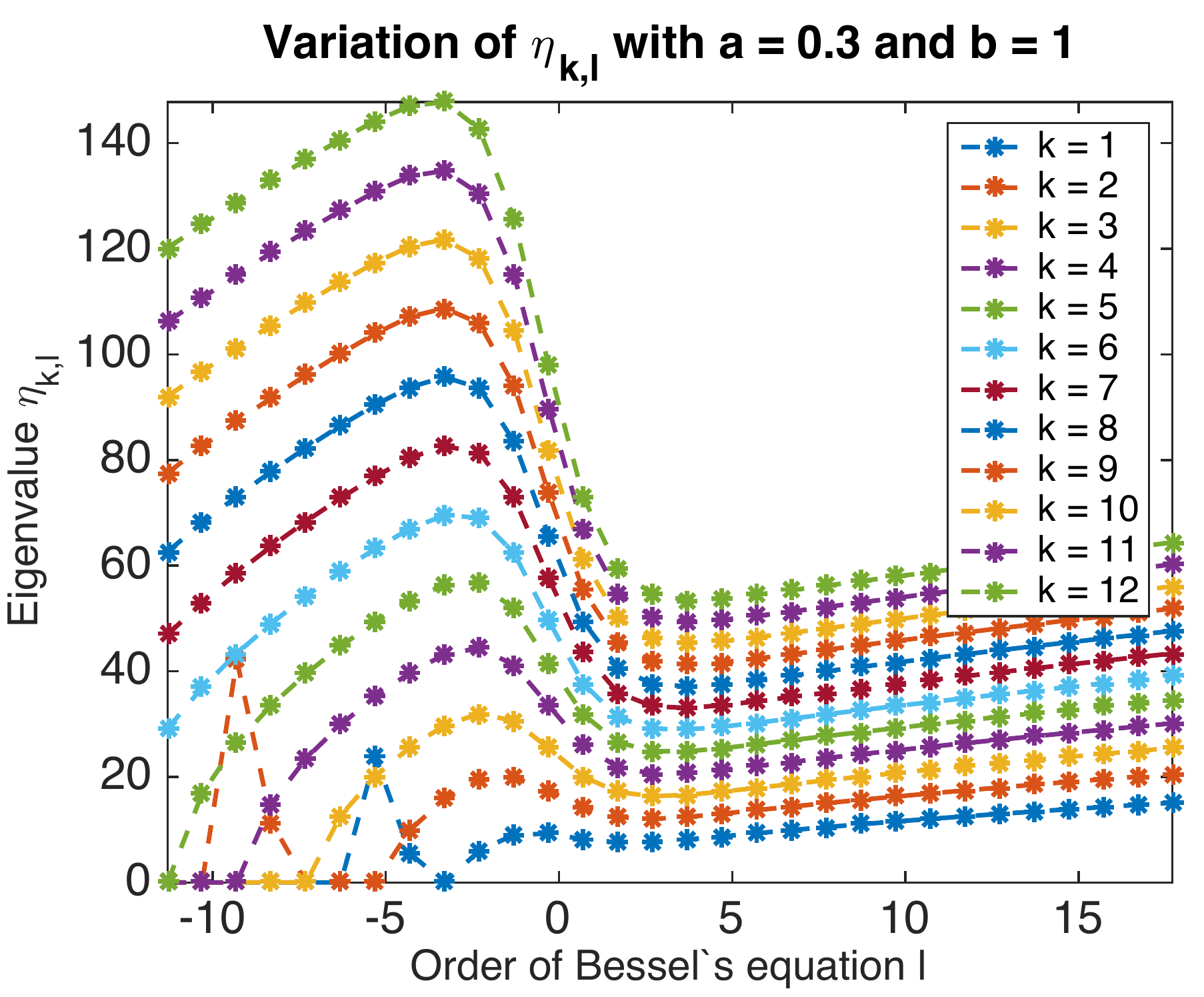}}
\subfigure[]{\includegraphics[width=0.32\textwidth]{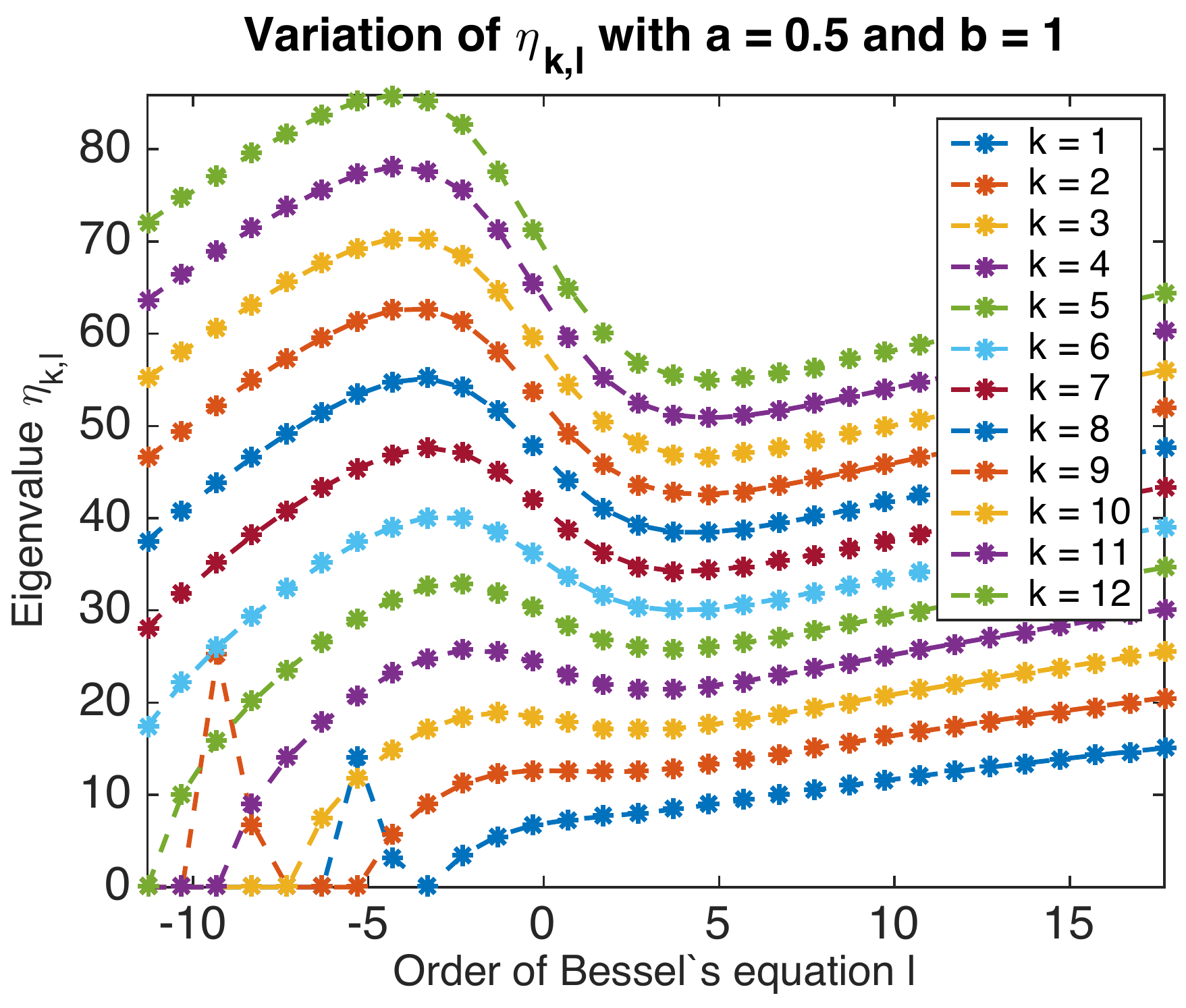}}
\subfigure[]{\includegraphics[width=0.49\textwidth]{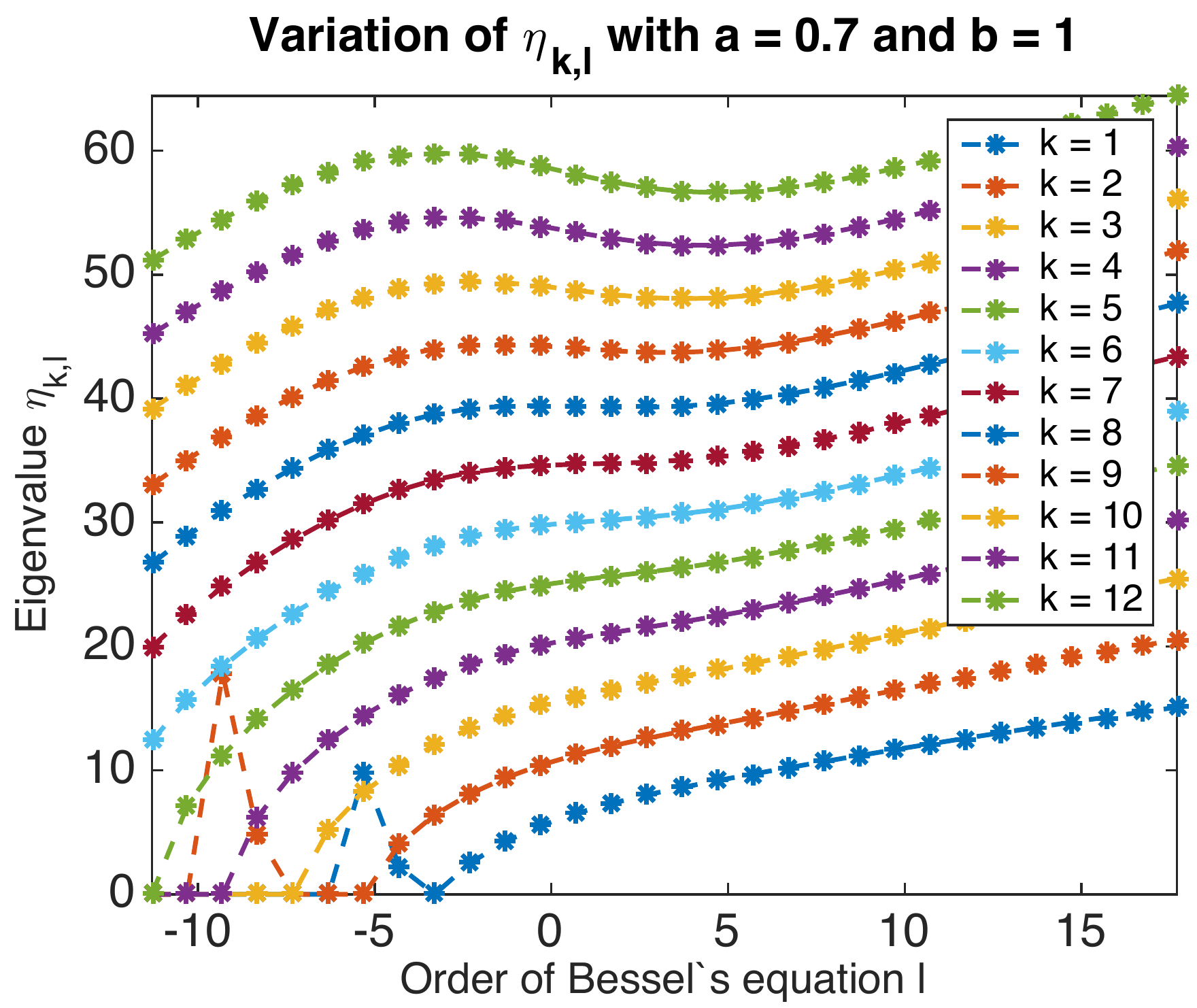}}
\subfigure[]{\includegraphics[width=0.49\textwidth]{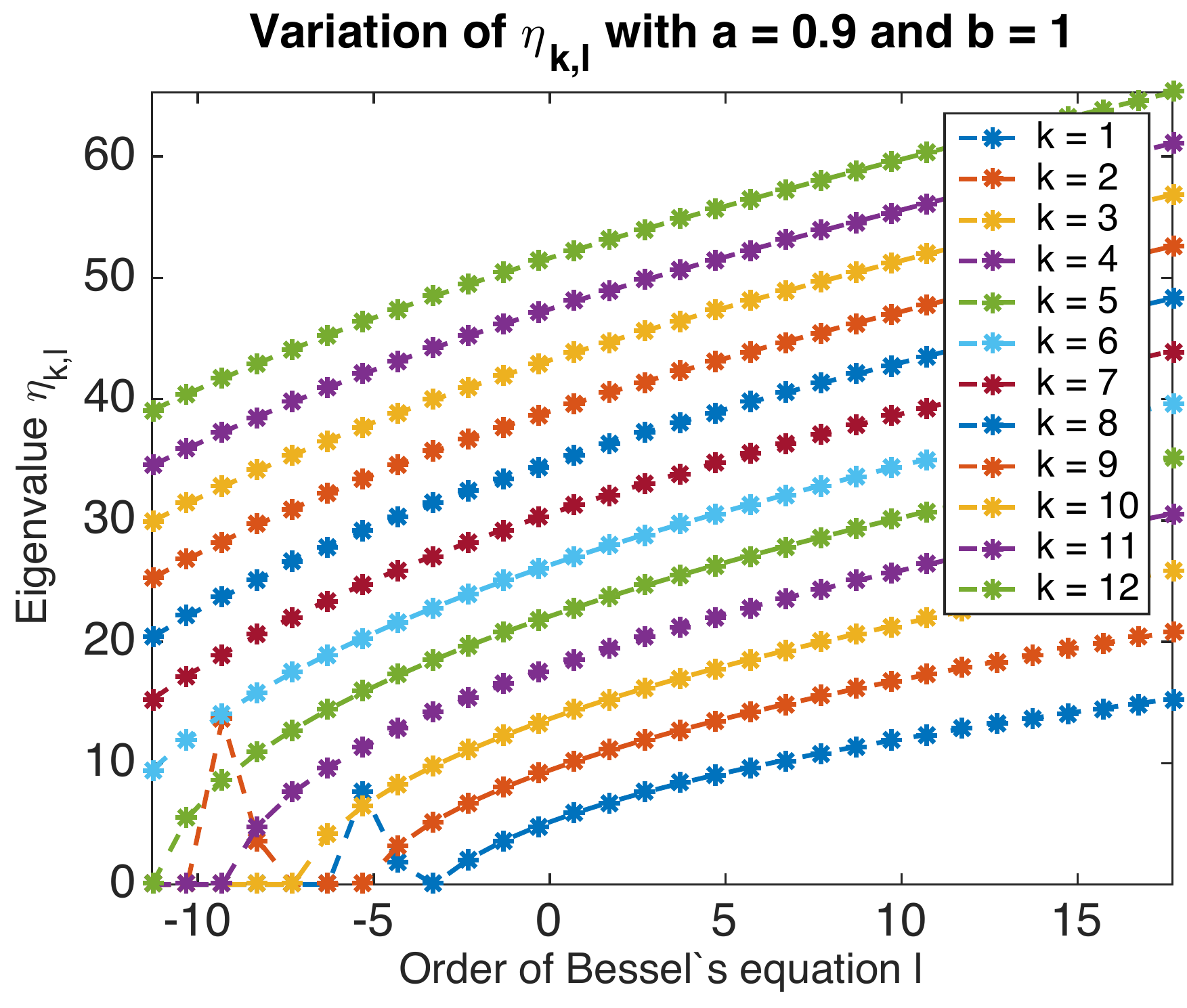}}
\end{center}
	\caption{Variation of the semi-discrete spectrum with respect to the associated order of the  Bessel's equation}
	\label{varmesh}
\end{figure}
\subsection{Stability matrix and the characteristic polynomial}
We employ the solution proposed by Theorem \ref{theorem1} of the eigenvalue problem (\ref{eigen}) and adopt an application of separation of variables to write the analytical solutions to the linearised approximation of problem (\ref{polarsystem}) as an infinite sum that consists of the product of the eigenfunctions of $\triangle_p$, namely $w(r,\theta)=R_{k,l}(r)\Theta_l(\theta)$ and $T(t)=\exp{(\sigma_{k,l})}$. With bars omitted from the variables $\bar{u}$ and $\bar{v}$, one may write the solution to the linearised system (\ref{vect}) in the form 
\[
u(r,\theta,t) = \sum_{k=0}^\infty U_{k,l}\exp{(\sigma_{k,l} t)}R_{k,l}(r)\Theta_l(\theta) \]
and \[
v(r,\theta,t) = \sum_{k=0}^\infty V_{k,l}\exp{(\sigma_{k,l} t)}R_{k,l}(r)\Theta_l(\theta),
\]
where $U_{k,l}$ and $V_{k,l}$ denote the coefficients of the terms that correspond to the eigenmodes of superposition. We substitute this form of solutions and the expressions of the uniform steady state namely $(u_s,v_s)=(\alpha+\beta,\frac{\beta}{(\alpha+\beta)^2})$ in system (\ref{vect}), which leads to a fully linearised version of (\ref{vect}) and provides a discrete two-dimensional algebraic eigenvalue problem \cite{book1, book7, book12} with $\sigma_{1,2}$ denoting the eigenvalues, through which the bifurcation analysis and parameter classification of system (\ref{polarsystem}) is extensively conducted.  

These eigenvalues are computed through solving the relevant characteristic polynomial which is written as
\begin{equation}
\left|\begin{array}{cc}
\gamma \frac{\beta-\alpha}{\beta+\alpha}-\eta_{k,l}^2 -\sigma& \gamma (\beta+\alpha)^2  \\
-\gamma \frac{2\beta}{\beta+\alpha}&-\gamma(\beta+\alpha)^2-d\eta_{k,l}^2-\sigma
\end{array}\right|=0,
\label{vect4}
\end{equation}
and it can further be written in terms of the trace $\mathcal{T}(\alpha,\beta)$ and the determinant $\mathcal{D}(\alpha,\beta)$ of the stability matrix associated to (\ref{vect4}) in the form of a quadratic equation as 
\begin{equation}
\sigma^2+\mathcal{T}(\alpha,\beta)\sigma+\mathcal{D}(\alpha,\beta)=0,
\label{characpol}
\end{equation} 
where $\mathcal{T}(\alpha,\beta)$ and $\mathcal{D}(\alpha,\beta)$ are given by 
\begin{equation}\begin{cases}
\mathcal{T}(\alpha, \beta) =& \gamma \big(\frac{\beta - \alpha  - (\beta +\alpha)^3}{\beta +\alpha}\big)-(d+1)\eta_{k,l}^2,\\
\mathcal{D}(\alpha, \beta) =& \big(\gamma\frac{\beta-\alpha}{\beta+\alpha}-\eta_{k,l}^2\big)\big(-\gamma(\beta +\alpha)^2-(d+1)\eta_{k,l}^2\big)+2\gamma^2\beta(\beta+\alpha).
\end{cases}\label{detrace}\end{equation}
The roots of equation (\ref{characpol}) are given by $\sigma_{1,2}=\frac{\mathcal{T}\pm\sqrt{\mathcal{T}^2-4\mathcal{D}}}{2}$ in terms of $\mathcal{T}$ and $\mathcal{D}$. If $\sigma_{1,2}\in\mathbb{R}$ then the stability of the uniform steady state $(u_s,v_s)$ is determined by the signs of $\sigma_{1,2}$. If $\sigma_{1,2}$ turns out to be a pair of complex conjugate values i.e. $\sigma_{1,2}\in\mathbb{C}\backslash\mathbb{R}$, then it is the sign of the real part $\text{Re}(\sigma)$ that determines the stability of the uniform steady state $(u_s,v_s)$. If $\sigma_{1,2}$ is real, then the uniform steady state undergoes unstable behaviour if at least $\sigma_1$ or $\sigma_2$ has positive sign. Therefore, when the roots of the characteristic polynomial (\ref{characpol}) are purely real, then the existence of one positive root suffices to decide that the corresponding uniform steady state is unstable. The uniform steady state $(u_s,v_s)$ under this circumstance is stable if and only if both $\sigma_{1,2}$ possess negative signs. Therefore, in order to encapsulate all the possibilities for the stability and types of the uniform steady state $(u_s,v_s)$ in light of parameters $\alpha$ and $\beta$, it is necessary to consider the cases when  $\sigma_{1,2}\in\mathbb{C}\backslash\mathbb{R}$ and $\sigma_{1,2}\in\mathbb{R}$. The parameter space is rigorously analysed and fully classified under both cases to determine regions on the plane $(\alpha,\beta)\in \mathbb{R}_+^2$ that correspond to different types of dynamical predictions of system (\ref{polarsystem}). Furthermore, it is investigated to find how the classification of the parameter spaces is influenced by the variation of the diffusion parameter $d$. In light of such classification, the analysis is further extended to explore the effects of domain size in particular the thickness $\rho=b-a$ corresponding to the annular domain $\Omega$ on the spatial and temporal bifurcation of the dynamical system (\ref{polarsystem}).
\section{Partitioning curves and bifurcation analysis}\label{parbif}
Regions on the admissible parameter space i.e. $(\alpha,\beta)\in \mathbb{R}_+^2$ that correspond to the stability and types of the uniform steady state is extensively explored and the equations that determine the partition of such classification within the admissible parameter space are obtained and an analytical study is performed on them to find how the domain-size (thickness $\rho=b-a$) influences the bifurcation predictions of the dynamics corresponding to system (\ref{polarsystem}). 
Considering the roots of the quadratic polynomial (\ref{characpol}), which is given by $\sigma_{1,2}=\frac{\mathcal{T}\pm\sqrt{\mathcal{T}^2-4\mathcal{D}}}{2},$ we note that the partitioning of the admissible parameter space namely $(\alpha,\beta)\in \mathbb{R}_+^2$ is determined by two curves, one of which satisfies the equation $\mathcal{T}^2(\alpha,\beta)-4\mathcal{D}(\alpha,\beta)=0$ and the other one satisfying $\mathcal{T}(\alpha,\beta)=0$ given that $\mathcal{D}(\alpha,\beta)>0$. We start with the curve satisfying $\mathcal{T}^2(\alpha,\beta)-4\mathcal{D}(\alpha,\beta)=0$ on the parameter plane $(\alpha,\beta)\in \mathbb{R}_+^2$ that forms a boundary for the region that corresponds to eigenvalues $\sigma_{1,2}$ containing non-zero imaginary part from that which corresponds to a pair of purely real $\sigma_{1,2}$. 
We proceed with this setup in mind and write the equation of the curve that forms the boundary between regions corresponding to $\sigma\in\mathbb{R}$ and that which corresponds to $\sigma_{1,2}\in\mathbb{C}\backslash\mathbb{R}$. Such an equation is of the form 
\begin{equation}\begin{split}
\Big(\gamma \frac{\beta - \alpha  - (\beta +\alpha)^3}{\beta +\alpha}-(d+1)\eta_{k,l}^2\Big)^2=&4\Big(\big(\gamma\frac{\beta-\alpha}{\beta+\alpha}-\eta_{k,l}^2\big)\big(-\gamma(\beta +\alpha)^2\\
&-(d+1)\eta_{k,l}^2\big)+2\gamma^2\beta(\beta+\alpha)\Big),
\label{part1}
\end{split}\end{equation}
where $\eta^2_{k,l}$ satisfies (\ref{eigenvalue1}). Equation (\ref{part1}) is numerically solved in Section \ref{main}, where a numerical method using polynomials \cite{paper38, paper47} is employed to find combinations of $\alpha,\beta\in\mathbb{R}_+$ on the plane $(\alpha,\beta)\in \mathbb{R}_+^2$, that lie on the curve satisfying (\ref{part1}). 
Note that the curve satisfying (\ref{part1}) also forms the boundary on the parameter space for regions of spatial and temporal bifurcations. 
Any instability occurring from choosing the parameters from the side of the partitioning curve where $\sigma_{1,2}$ are a pair of real values, it is predicted to evolve to a steady state of spatial variation of Turing type, hence the obtained pattern from system (\ref{polarsystem}) evolves with a globally stable and invariant behaviour in time. Comparing this to the instability that arises from parameters on the side where $\sigma_{1,2}\in\mathbb{C}\backslash\mathbb{R}$, it is predicted to evolve periodically in time, therefore the resulting steady state is expected to bifurcate into a spatial pattern in time. 
The second equation that forms the partition on the admissible parameter space is $\mathcal{T}(\alpha,\beta)=0$, given that $\mathcal{D}(\alpha,\beta)>0$, which can be written as
\begin{equation}
\gamma\big(\beta - \alpha  - (\beta +\alpha)^3\big)=(\alpha+\beta)(d+1)\eta_{k,l}^2.
\label{part2}
\end{equation}
Note that the solutions to (\ref{part1}) and (\ref{part2}) offer a full classification of the admissible parameter space in the sense that it predicts the dynamical behaviour exhibited by system (\ref{polarsystem}) for every possible choice of $(\alpha,\beta)\in\mathbb{R}_+^2$.
\subsection{Analysis for the case of complex eigenvalues}
We analyse the real part of $\sigma_{1,2}$, when $\sigma_{1,2}$ is a complex conjugate pair, which occurs if and only if $(\alpha,\beta)$ satisfies the inequality
\begin{equation}
\mathcal{T}^2(\alpha,\beta)-4\mathcal{D}(\alpha,\beta)<0.
\label{comp}
\end{equation}
Given that $(\alpha,\beta)$ satisfies (\ref{comp}), then the sign of $\text{Re}(\sigma_{1,2})$ determines the stability of the uniform steady steady state $(u_s,v_s)$, which is the expression
\begin{equation}
\text{Re}(\sigma_{1,2})=\frac{1}{2}\Big(\gamma\frac{\beta-\alpha-(\beta +\alpha)^3}{\beta+\alpha}-(d+1)\eta_{k,l}^2\Big).
\label{real}
\end{equation}
If the sign of the right hand-side of (\ref{real}) is negative, under the restriction (\ref{comp}), then the dynamics of system (\ref{polarsystem}) are forbidden from temporal bifurcation for all choices of $(\alpha,\beta)\in\mathbb{R}_+^2$. Therefore, with (\ref{comp}) satisfied and the RHS of (\ref{real}) positive, if the dynamics of system (\ref{polarsystem}) do exhibit diffusion-driven instability, it must be a strictly spatially periodic behaviour only, which uniformly converges to a temporal steady state of Turing type, consequently one obtains spatial pattern that is invariant in time. The sign of the expression given in (\ref{real}) is further investigated to derive from it, relations between the parameter $\rho=b-a$, which controls the domain size and the reaction-diffusion rates denoted by $\gamma$ and $d$ respectively. Given that assumption (\ref{comp}) is satisfied then the sign of expression (\ref{real}) is negative if parameters $\alpha$, $\beta$, $\gamma$ and $d$ satisfy the inequality 
\begin{equation}
\frac{\beta - \alpha  - (\beta +\alpha)^3}{\beta +\alpha}<\frac{(d+1)\eta_{k,l}^2}{\gamma},
\label{comp1}
\end{equation}
with $\eta_{n,k}^2$ defined by (\ref{eigenvalue1}). Note that the expression on the left hand-side of (\ref{comp1}) is a bounded quantity by the constant value of 1 \cite{paper38}, for all the admissible choices of $(\alpha,\beta)\in\mathbb{R}_+^2$. We aim to investigate inequality (\ref{comp1}), so that we can establish a restriction on the parameter $\rho=b-a$ in terms of everything else that ensures the real part of $\sigma_{1,2}$ to be negative, which is equivalent to imposing a condition that guarantees global temporal stability in the dynamics. For this to hold we need to incorporate the parameter $\rho$ defining the quantity $b-a$ into the expression for $\eta_{k,l}^2$. This expression can be written in terms of parameter $\rho$ and $a$, where we replace parameter $b$ by $\rho+a$, and using $\rho=b-a$, then (\ref{eigenvalue1}) takes the form
\begin{equation}
\eta_{k,l}^2=f(\rho)\frac{4(2k+1)(l+2k+1)(l+4k)}{(l+4k+2)}.
\label{comp2}
\end{equation}
The domain-dependent weighting function $f(\rho)$ in (\ref{comp2}) is given by 
\begin{equation}
f(\rho)=\frac{a^{l-1}+(\rho+a)^{l-1}}{a^{l+1}+(\rho+a)^{l+1}}.
\label{weight1}
\end{equation}
The boundedness of the expression on the left hand-side of (\ref{comp1}) by the constant value of 1 \cite{paper38}, entails that the inequality given by (\ref{comp1}) can be written as
\begin{equation}
\gamma < f(\rho)\frac{4(d+1)(2k+1)(l+2k+1)(l+4k)}{(l+4k+2)}.
\label{comp3}
\end{equation}
This can further be studied by investigating the behaviour of $f(\rho)$ in (\ref{weight1}), since $f(\rho)$ has a weighting effect on the magnitude of $\eta_{k,l}^2$. In order to encapsulate all the possibilities for $f(\rho)$, we first assert that $f(\rho)$ is a monotonically decreasing function, irrespective of the choice of $a$ and $b$, given that $0<a<b$. The behaviour of $f(\rho)$ is similar to that of $\frac{1}{\rho^2}$ with the limiting case as $a\rightarrow 0$, which follows from the fact that $(\rho+a)^{l+1}$ resides in the denominator and $(\rho+a)^{l-1}$ resides in the numerator, therefore asymptotically one may write that $f(\rho+\epsilon)< f(\rho), \forall \epsilon >0$, and hence $\eta_{k,l}^2 \rightarrow 0, \text{ as } \rho \rightarrow \infty$. This claim is verified numerically and shown in Figure \ref{nodal} (b) where, $f(\rho)$ is simulated for various values of $a$ shown in the respective legend.
\begin{figure}[ht]
\begin{center}
	\subfigure[]{\includegraphics[width=0.49\textwidth]{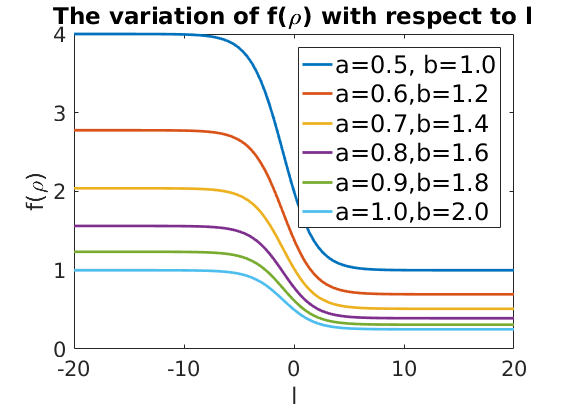}}
	\subfigure[]{\includegraphics[width=0.49\textwidth]{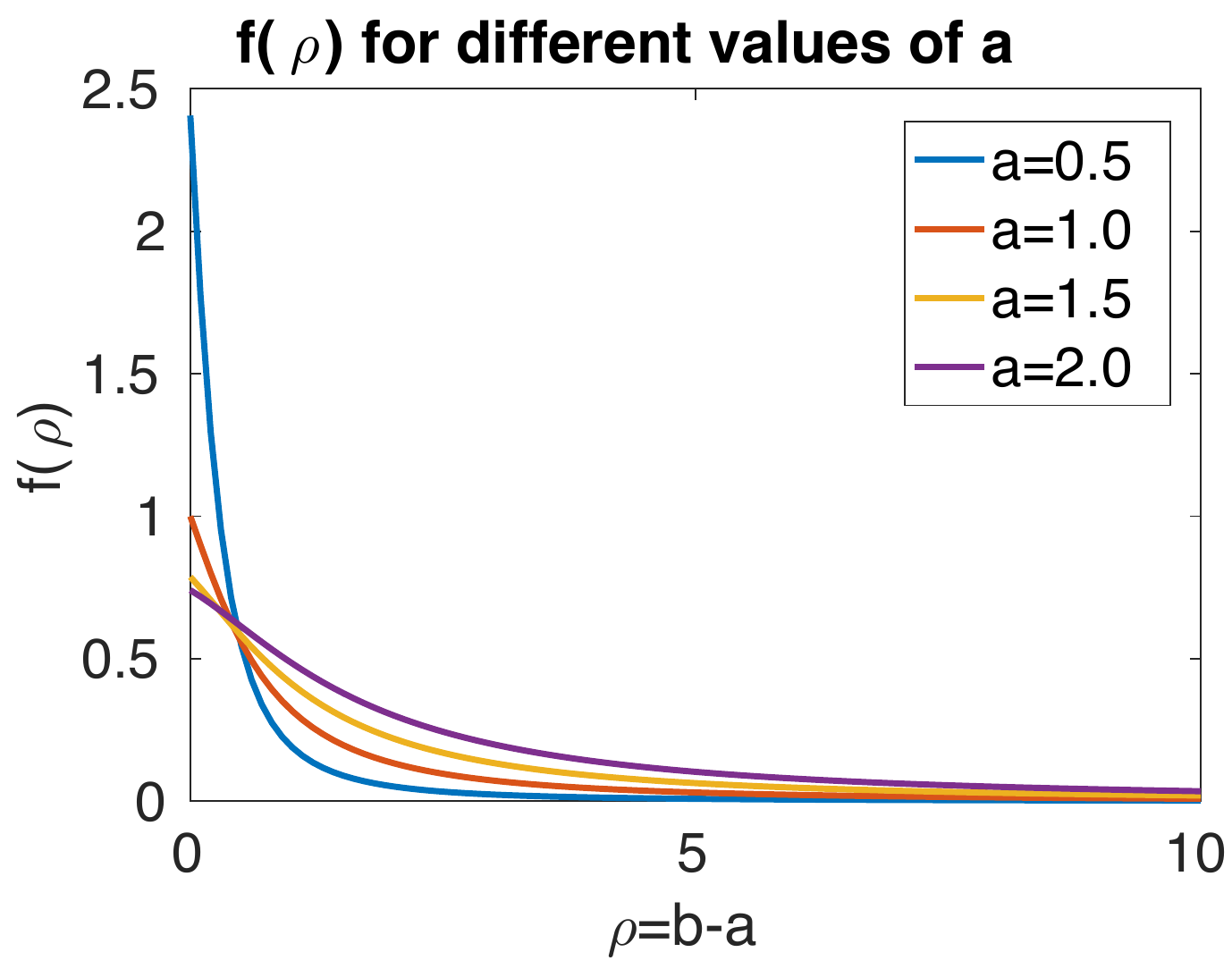}}
\end{center}
	\caption{The variation of the weighting function $f(\rho)$ with respect to the order of Bessel's equation $l$ in (a) and with respect to $\rho$ in (b)}
	\label{nodal}
\end{figure}
The analysis of $f(\rho)$ and the variation of the spectrum of $\triangle_p$ with respect to the associated order of Bessel's equation $l$ (shown in Figure \ref{varmesh}) on the non-compact domain $\Omega$ suggests that two asymptotic  cases require independent focus for the validity of inequality (\ref{comp3}). Two cases correspond to the two local suprema attained by $f(\rho)$ with respect to $l<0$ and $l>0$ respectively. 

From numerical investigation of $f(\rho)$ with respect to the associated order of Bessel's equation $l$ shown in Figure \ref{nodal} (a), it can be found that the two suprema for $f(\rho)$ with $l<0$ and $l>0$ are respectively given by
\begin{equation}\label{suprema}
\begin{split}
\sup_{0>l\in\mathbb{R}\backslash\frac{1}{2}\mathbb{Z}}f(\rho)=\lim_{l\rightarrow - \infty} f(\rho) = \frac{2}{a(\rho+a)},\\ \sup_{0<l\in\mathbb{R}\backslash\frac{1}{2}\mathbb{Z}}f(\rho)=\lim_{l\rightarrow 0} f(\rho) = \frac{1}{a(\rho+a)}.
\end{split}
\end{equation}
We proceed to employ these asymptotic upper bounds on the weighting function $f(\rho)$ to obtain the necessary conditions in each of the limiting cases for $l$, that ensures the validity of inequality (\ref{comp3}). We note that requiring (\ref{comp3}) to be valid for each one of the two cases corresponding to $l \rightarrow -\infty$ and $l\rightarrow 0^+$, give rise to a different condition on $\rho$. Starting with the case $\lim_{l\rightarrow -\infty} f(\rho)=\frac{2}{a(\rho+a)}$ by substituting in (\ref{comp3}) and rearranging, we obtain that for (\ref{comp3}) to be valid with $0>l\in\mathbb{R}\backslash\frac{1}{2}\mathbb{Z}$, the parameter $\rho=b-a$ must satisfy
\begin{equation}
\rho < \frac{8(d+1)(2k+1)(l+2k+1)(l+4k)-\gamma a^2(l+4k+2)}{\gamma a (l+4k+2)}.
\label{asymp1}
\end{equation}
Using similar approach by substituting  $\lim_{l\rightarrow 0^+} f(\rho)=\frac{1}{a(\rho+a)}$ in (\ref{comp3}), we can find for positive values of $l\in\mathbb{R}\backslash\frac{1}{2}\mathbb{Z}$ the associated condition on the parameter $\rho$ in the form given by
\begin{equation}
\rho < \frac{4(d+1)(2k+1)(l+2k+1)(l+4k)-\gamma a^2(l+4k+2)}{\gamma a (l+4k+2)}.
\label{asymp2}
\end{equation}
It is worth noting that (\ref{asymp1}) and (\ref{asymp2}) are sharp conditions on the parameter $\rho$ controlling the area of $\Omega$ for positive and negative order of the associated Bessel's equation given by $0>l\in\mathbb{R}\backslash\frac{1}{2}\mathbb{Z}$ and $0<l\in\mathbb{R}\backslash\frac{1}{2}\mathbb{Z}$ respectively. However, since condition (\ref{asymp1}) corresponds to the global supremum for the weighting function $f(\rho)$, therefore, without loss of generality, condition (\ref{asymp1}) can be represented to ensure the validity of (\ref{comp3}) for all the admissible choices of $l\in\mathbb{R}\backslash\frac{1}{2}\mathbb{Z}$. 
Condition (\ref{asymp1}) conversely implies that when $\sigma_{1,2}\in\mathbb{C}\backslash\mathbb{R}$ then $\text{Re}(\sigma_{1,2})>0$ if parameter $\rho$ satisfies 
\begin{equation}
\rho \geq \frac{8(d+1)(2k+1)(l+2k+1)(l+4k)-\gamma a^2(l+4k+2)}{\gamma a (l+4k+2)},
\label{asymp3}
\end{equation}
which consequently means that if (\ref{asymp3}) is satisfied then the steady state $(u_s,v_s)$ undergoes temporal diffusion-driven instability leading to a spatial pattern bifurcating in time. This statement is formalised in Theorem \ref{theorem2}. 
\begin{theorem}[Hopf or transcritical bifurcation]
	Let $u$ and $v$ satisfy the non-dimensional reaction-diffusion system with {\it activator-depleted} reaction kinetics (\ref{polarsystem}) on a non-compact two dimensional (shell) domain $\Omega \subset \mathbb{R}^2$ with thickness $\rho$ and positive real parameters $\gamma$, $d$, $\alpha$ and $\beta$.
	For the system to exhibit Hopf or transcritical bifurcation in the neighbourhood of the unique steady state $(u_s,v_s)=\big(\alpha+\beta, \frac{\beta}{(\alpha+\beta)^2}\big)$, the necessary condition on the thickness $\rho$ of $\Omega \subset \mathbb{R}^2$ is that it must be sufficiently large satisfying (\ref{asymp3}) with
	$l\in\mathbb{R}\backslash\frac{1}{2}\mathbb{Z}$ denoting the associated order of the Bessel's equations and $k$ is any positive integer. In (\ref{asymp3}) the parameter $a$ denotes the radius of the inner boundary of $\Omega$.
	\label{theorem2}
\end{theorem}
\begin{proof}
	The proof of this theorem takes an identical approach to that presented in \cite{paper38, paper47}, except the use of the expression (\ref{comp2}) for $\eta_{k,l}^2$ in the final step of the proof.$\square$	
\end{proof}
The numerical verification of Theorem \ref{theorem2} is presented in Section \ref{main} to show that regions corresponding to Hopf and transcritical bifurcations exist under condition (\ref{asymp3}) on the controlling parameter for the domain-size, which is $\rho$. It is also numerically demonstrated that no choice of parameters $(\alpha,\beta)\in\mathbb{R}_+^2$ exist that would lead to Hopf or transcritical bifurcation, when condition (\ref{asymp3}) is violated.
\subsection{Analysis for the case with real eigenvalues}\label{analreal}
The necessary and sufficient condition on the discriminant $\mathcal{T}^2(\alpha,\beta)-4\mathcal{D}(\alpha,\beta)$ for $\sigma_{1,2}$ to be a pair of real values is that $\mathcal{T}^2(\alpha,\beta) \geq 4\mathcal{D}(\alpha,\beta)$.
We first analyse the equal case and consider
\begin{equation}
\mathcal{T}^2(\alpha,\beta) = 4\mathcal{D}(\alpha,\beta),
\label{condequal}
\end{equation}
which vanishes the discriminant, therefore, $\sigma_{1,2}$ become a pair of repeated real values given by
\begin{equation}
\sigma_1=\sigma_2 = \frac{1}{2}\Big(\gamma\frac{\beta - \alpha  - (\beta +\alpha)^3}{\beta +\alpha}-(d+1)\eta_{k,l}^2\Big).
\label{repeated}
\end{equation}
where we substitute in (\ref{repeated}) for $\eta_{k,l}^2$ the expression given in (\ref{comp2}) with both the local suprema (\ref{suprema}) of the weighting function $f(\rho)$ for $l<0$ and $l>0$.
When $\alpha$ and $\beta$ satisfy condition (\ref{condequal}), the stability of the steady state is determined by the sign of the root itself. 
The supremum of the weighting function $f(\rho)$ for $l<0$ is substituted in (\ref{repeated}), which can be easily shown to be negative if parameter $\rho$ satisfies the inequality 
\begin{equation}
\begin{split}
\rho<&\frac{8(d+1)(\beta+\alpha)(2k+1)(l+2k+1)(l+4k)}{\gamma a (\beta-\alpha-(\alpha+\beta)^3)(l+4k+2)}-a.
\label{repeatneg}
\end{split}
\end{equation}
Otherwise, the repeated root is positive provided that $\rho$ satisfies
\begin{equation}
\begin{split}
\rho>&\frac{8(d+1)(\beta+\alpha)(2k+1)(l+2k+1)(l+4k)}{\gamma a (\beta-\alpha-(\alpha+\beta)^3)(l+4k+2)}-a.
\label{repeatpos}
\end{split}
\end{equation}
Similarly with (\ref{condequal}) satisfied and $l>0$, upon substituting $\sup_{0<l\in\mathbb{R}\backslash\frac{1}{2}\mathbb{Z}}f(\rho)=\frac{1}{a(\rho+a)}$ in (\ref{repeated}) and rearranging one can easily show that the sign of the expression given by (\ref{repeated}) is negative if $\rho$ satisfies
\begin{equation}
\rho<\frac{4(d+1)(\beta+\alpha)(2k+1)(l+2k+1)(l+4k)}{\gamma a (\beta-\alpha-(\alpha+\beta)^3)(l+4k+2)}-a.
\label{repeat1neg}
\end{equation}
Otherwise, the repeated root $\sigma_{1,2}$ given by (\ref{repeated}) is positive with $l>0$ if $\rho$ satisfies
\begin{equation}
\rho>\frac{4(d+1)(\beta+\alpha)(2k+1)(l+2k+1)(l+4k)}{\gamma a (\beta-\alpha-(\alpha+\beta)^3)(l+4k+2)}-a.
\label{repeat1pos}
\end{equation}
Further analysis of conditions (\ref{repeatneg}) and (\ref{repeatpos}) is required to ensure that parameter $\rho$ is not compared against a negative quantity. First we note that the only terms that can possibly invalidate inequalities (\ref{repeatneg}), (\ref{repeatpos}),(\ref{repeat1neg}) and (\ref{repeat1pos}) are in the denominator of the right hand-side, namely the expression $\beta-\alpha-(\beta+\alpha)^3$. Therefore, a restriction is required to be stated on this term to ensure that the radius $\rho$ of $\Omega$ is not compared against an imaginary number, such a restriction is 
\begin{equation}
\beta > \alpha+(\beta+\alpha)^3.
\label{rest}
\end{equation}
Furthermore, given that restriction (\ref{rest}) is satisfied then we note that the right hand-sides of inequalities (\ref{repeatneg}), (\ref{repeatpos}),(\ref{repeat1neg}) and (\ref{repeat1pos}) are ensured to be positive if parameter $a$ satisfies
\[
\frac{4(d+1)(\beta+\alpha)(2k+1)(l+2k+1)(l+4k)}{\gamma a (\beta-\alpha-(\alpha+\beta)^3)(l+4k+2)}>a.
\]

It must be noted that (\ref{rest}) is an identical restriction on the parameter choice obtained for the case of repeated real eigenvalues in the absence of diffusion \cite{paper38}. 
A reasonable intuition behind this comparison is that the sub-region on the admissible parameter plane that corresponds to complex eigenvalues with negative real parts must be bounded by curve (\ref{condequal}) subject to conditions (\ref{repeatneg}) and (\ref{repeat1neg}), outside of which every possible choice of parameters $\alpha$ and $\beta$ will guarantee the eigenvalues $\sigma_{1,2}$ to be a pair of distinct real values, which promotes the necessity to state Theorem \ref{theorem3}.
\begin{theorem}[Turing type diffusion-driven instability]
	Let $u$ and $v$ satisfy the non-dimensional reaction-diffusion system with {\it activator-depleted} reaction kinetics (\ref{polarsystem}) on a non-compact (shell) domain $\Omega \subset \mathbb{R}^2$ with inner radius $a$, thickness $\rho$ and positive real parameters $\gamma$, $d$, $\alpha$ and $\beta$.
	Given that the thickness $\rho$ of domain $\Omega \subset \mathbb{R}^2$ satisfies the inequality (\ref{asymp2})
	with $l\in\mathbb{R}_+\backslash\frac{1}{2}\mathbb{Z}$ denoting the associated order of the Bessel's equations and $k$ is any positive integer, then for all $\alpha, \beta \in \mathbb{R}_+$ in the neighbourhood of the unique steady state $(u_s,v_s)=\big(\alpha+\beta, \frac{\beta}{(\alpha+\beta)^2}\big)$ the diffusion-driven instability is restricted to Turing type only, forbidding the existence of Hopf and transcritical bifurcations.
	\label{theorem3}
\end{theorem}
\begin{proof}
	The strategy of this proof is mostly identical to that given in \cite{paper38, paper47}, therefore to avoid repetition of the details, the proof is omitted and the interested reader is referred to consult \cite{paper38, paper47}. The only difference in the proof of the current theorem is in the last step where an explicit representation of eigenvalues given in (\ref{comp2}) is substituted containing the domain controlling parameter $\rho$.$\square$ 
\end{proof}	
\section{Numerical classification of parameter spaces and partitioning curves}\label{main}
Equation (\ref{part1}) and (\ref{part2}) can be algebraically manipulated and rearranged as six and three degree polynomials in $\beta$ respectively, whose coefficients depend on $\alpha$, $d$, $\gamma$ and $\eta_{k,l}^2$. Let $\mathcal{P}_1$ and $\mathcal{P}_2$ respectively denote the six and three degree polynomials in $\beta$, which can be written in the form $\mathcal{P}_1(\beta)=\sum_{n=0}^6 C_n(\alpha)\beta^n$ and $\sum_{n=1}^3C_n(\alpha)\beta^n$ with $C_n$ denoting the coefficient of the term $\beta^n$. The region on the admissible parameter space $(\alpha,\beta)\in\mathbb{R}_+^2$ satisfying equation (\ref{part1}) and (\ref{part2}) is obtained by employing the method of polynomials presented in \cite{paper38,  paper47} to find the full classification of the corresponding bifurcation plane $(\alpha, \beta)\in\mathbb{R}_+^2$.
This algorithm is executed for five different values of $d$ to obtain the solutions of (\ref{part1}) and (\ref{part2}) under conditions (\ref{asymp3}) and (\ref{asymp2}) on the parameter that controls the size of $\Omega$, namely $\rho$. The shift and existence of the partitioning curves satisfying (\ref{part1}) and (\ref{part2}) are analysed subject to the variation of parameter $d$. 
Using condition (\ref{asymp3}) of Theorem \ref{theorem2}, the variation of the diffusion coefficient is analysed for five different values of $d$ and Figure \ref{realcomplex} (a) shows the shift of the solutions of (\ref{part1}). 
The five values of parameter $d$ are clearly indicated on the curves in Figure \ref{realcomplex}.
\begin{figure}[ht]
	\begin{center}
		\subfigure[Partitioning curves satisfying (\ref{part1}) \newline and condition (\ref{asymp3}) of Theorem \ref{theorem2}, that\newline partitions the region corresponding to\newline real $\sigma_{1,2}$ from that, which correspond-\newline s to complex conjugate pair of $\sigma_{1,2}$]{\includegraphics[width=0.49\textwidth]{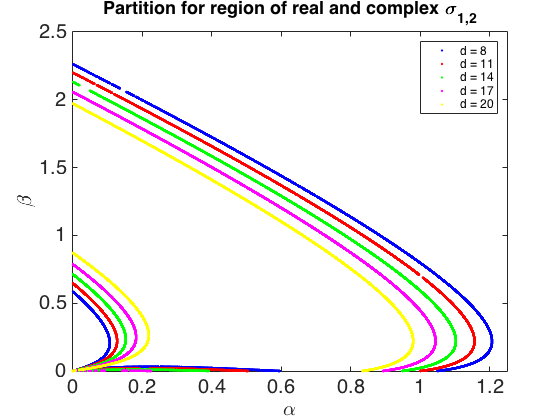}}
		\subfigure[Partitioning curves satisfying (\ref{part1}) and condition (\ref{asymp2}) of Theorem \ref{theorem3}, that partitions the region corresponding to real pair of $\sigma_{1,2}$ from that, which corresponds to complex conjugate pair of $\sigma_{1,2}$]{\includegraphics[width=0.49\textwidth]{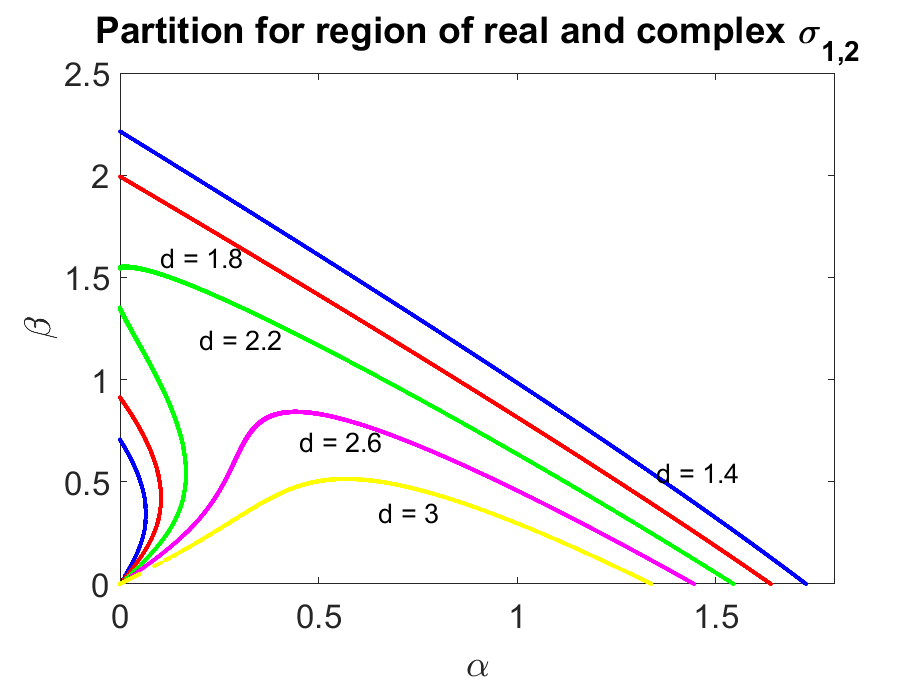}}
	\end{center}
	\caption{The effect of varying $d$ on the solution curves satisfying (\ref{part1}), where $\rho$ is used according to conditions (\ref{asymp3}) and (\ref{asymp2}) of Theorems \ref{theorem2} and \ref{theorem3} respectively.}
	\label{realcomplex}
\end{figure}
It would be reasonable to use exactly the same range for the variational values of $d$ under both conditions (\ref{asymp3}) and (\ref{asymp2}), however, it must be noted that, when the domain size $\rho$ is restricted by (\ref{asymp2}) then the same values used for varying $d$ in Figure \ref{realcomplex} (a) invalidate inequality (\ref{asymp2}), therefore, the range of variational values for parameter $d$ under condition (\ref{asymp2}) is significantly smaller. 
Figure \ref{realcomplex} (b) shows the variation of parameter $d$ using the values indicated in the respective legend. 
Using the method of trial and error proposed in \cite{paper38} it is determined that the sides shown in Figures \ref{complex} (a) and (b) are the regions corresponding to $\sigma_{1,2}\in\mathbb{C}\backslash\mathbb{R}$. Each stripe corresponding to a distinct value of $d$ is colour coded and denoted by capital alphabets in the form of a set containing all the points corresponding to a specific colour. These sets of points are referred to from Table \ref{table1} to present and summarise the quantitative analysis of the current numerical simulations of the admissible parameter spaces. 
\begin{figure}[ht]
	\begin{center}
		\subfigure[Shift of regions corresponding to comp-\newline lex $\sigma_{1,2}$ and subject to condition (\ref{asymp3}) of T-\newline heorem \ref{theorem2}]{\includegraphics[width=0.49\textwidth]{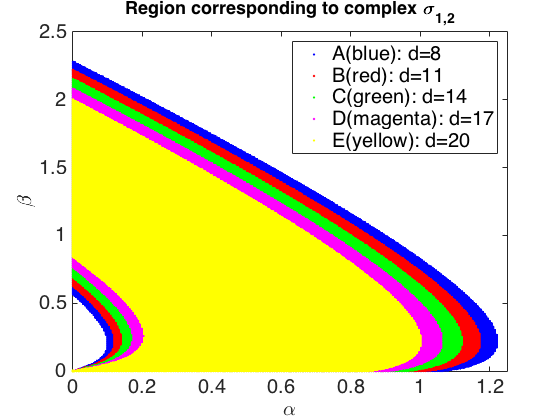}}
		\subfigure[Shift of regions corresponding to complex $\sigma_{1,2}$ and subject to condition (\ref{asymp2}) of Theorem \ref{theorem3}]{\includegraphics[width=0.49\textwidth]{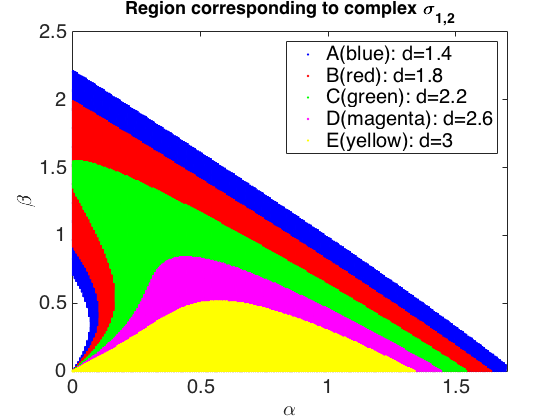}}
	\end{center}
	\caption{The shift in parameter spaces corresponding to complex $\sigma_{1,2}$ as a result of varying $d$}
	\label{complex}
\end{figure}
Through further study of regions corresponding to complex values for $\sigma_{1,2}$ in Figure \ref{complex}, using the solution of (\ref{part2}), it is numerically verified that a sub-partition only exists if the value of $\rho$ satisfies condition (\ref{asymp3}), with respect to the values of $d$ and $\gamma$. This is a numerical demonstration of the claim proposed by Theorem \ref{theorem2}. If the values of $d$ and/or $\gamma$ are changed such that $\rho$ no longer satisfies condition (\ref{asymp3}), it causes to vanish the existence of a sub-partition, within the region corresponding to complex eigenvalues $\sigma_{1,2}$, which is in agreement with Theorem \ref{theorem3}. 
Figure \ref{complexs} shows the regions on the admissible parameter spaces corresponding to complex $\sigma_{1,2}\in\mathbb{C}\backslash\mathbb{R}$ with negative real part. It can be noted that Figure \ref{complexs} (b) portrays exactly the same spaces as shown in Figure \ref{complex} (b), which is a further verification of Theorem \ref{theorem3}, namely, when $\rho$ satisfies condition (\ref{asymp2}), then for no choice of $\alpha,\beta \in \mathbb{R}_+$ the complex eigenvalue $\sigma_{1,2}$ can have a positive real part.
In this case we can only obtain a pattern of spots or stripes, with spatial periodicity.  If a condition on $\rho$ is set so that it is large enough to exceed the value on the right hand-side of condition (\ref{asymp2}), i.e. $\rho$ satisfying (\ref{asymp3}), only then a sub-partition can emerge within the admissible parameter space corresponding to $\sigma_{1,2}\in\mathbb{C}\backslash\mathbb{R}$. This can be observed by comparing Figure \ref{complexs} (a) with Figure \ref{complex} (a). Figure \ref{comufig} (b) shows the emergence of these curves that partition the region corresponding to $\sigma_{1,2}\in\mathbb{C}\backslash\mathbb{R}$. Recalling that if a sub-partition in the regions indicated by Figure \ref{complex} exist, then the corresponding partitioning curves must satisfy (\ref{part2}), which resembles the values of the parameter space that causes the real part of $\sigma_{1,2}$ to become zero when it is a pair of complex conjugate values. Therefore, on these curves the uniform steady state $(u_s,v_s)$ undergoes transcritical bifurcation. Figure \ref{comufig} (a) shows a shift in the region of the parameter spaces that corresponds to Hopf bifurcation for the same variation in the value of $d$ as used in Figure \ref{complex}. It is also worth noting that with increasing values of $d$, the parameter spaces corresponding to Hopf bifurcation gradually decrease. This is in agreement with the mathematical reasoning behind Theorem \ref{theorem2}, because as the value of $d$ is increased, one gets closer to the violation of the necessary condition (\ref{asymp3}) for the existence of regions for Hopf bifurcation. 
\begin{figure}[ht]
	\begin{center}
		\subfigure[Shift of regions corresponding to comp-\newline lex $\sigma_{1,2}$ with negative real part and subje-\newline ct to condition (\ref{asymp3}) of Theorem \ref{theorem2}]{\includegraphics[width=0.49\textwidth]{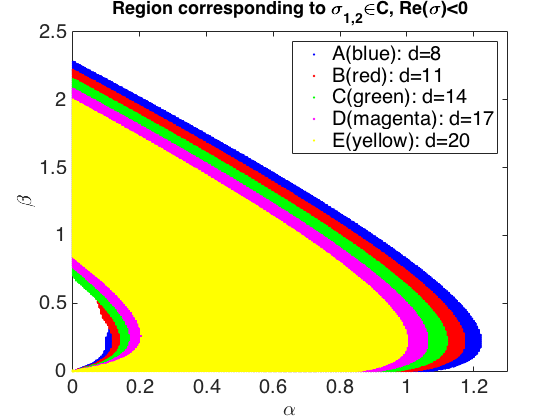}}
		\subfigure[Shift of regions corresponding to complex $\sigma_{1,2}$ with negative real parts and subject to condition (\ref{asymp2}) of Theorem \ref{theorem3}]{\includegraphics[width=0.49\textwidth]{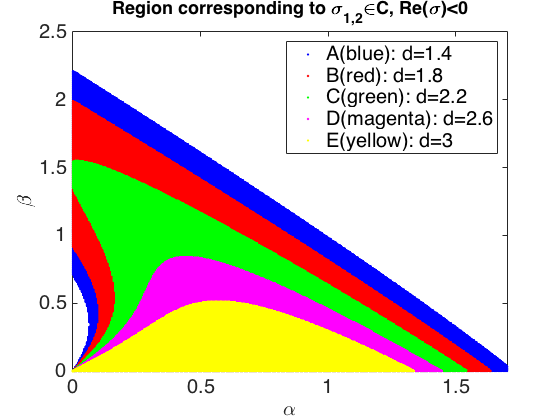}}
	\end{center}
	\caption{The shift in parameter spaces corresponding to complex $\sigma_{1,2}$ with negative real parts as a consequence of varying $d$}
	\label{complexs}
\end{figure}
\begin{figure}[ht]
	\begin{center}
		\subfigure[Shift of regions corresponding to comp-\newline lex $\sigma_{1,2}$ with positive real part and $\rho$  restr-\newline icted to condition (\ref{asymp3}) of Theorem \ref{theorem2}. Wit-\newline h condition (\ref{asymp2}) no region exists for Hopf b-\newline ifurcation]{\includegraphics[width=0.49\textwidth]{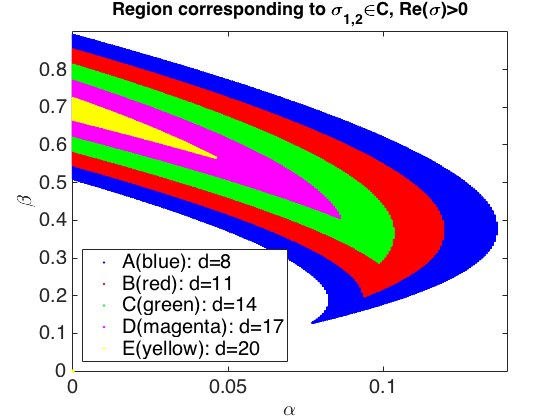}}
		\subfigure[Shift in the location of curves for five different values of $d$, on which the values of $\sigma_{1,2}$ are pure imaginary, and $\rho$ is restricted to condition (\ref{asymp3}) of Theorem \ref{theorem2}]{\includegraphics[width=0.49\textwidth]{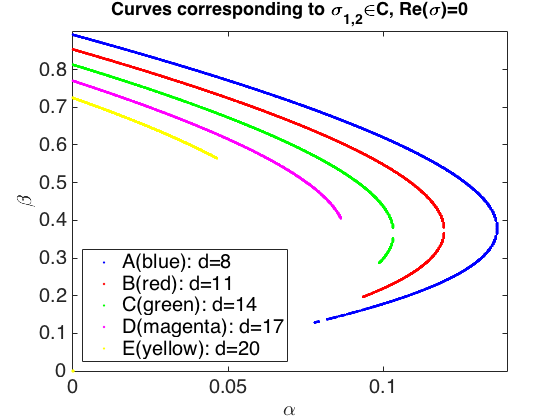}}
	\end{center}
	\caption{Regions in (a) correspond to Hopf bifurcation and curves in (b) correspond to transcritical bifurcation under condition (\ref{asymp3}). When $\rho$ satisfies condition (\ref{asymp2}), then no values of $\alpha$ and $\beta$ give rise to temporal instability in system (\ref{polarsystem})}
	\label{comufig}
\end{figure}
\begin{figure}[ht]
	\begin{center}
		\subfigure[Shift of regions corresponding to real \newline $\sigma_{1,2}$ and subject to condition (\ref{asymp3}) of The-\newline orem \ref{theorem2}]{\includegraphics[width=0.49\textwidth]{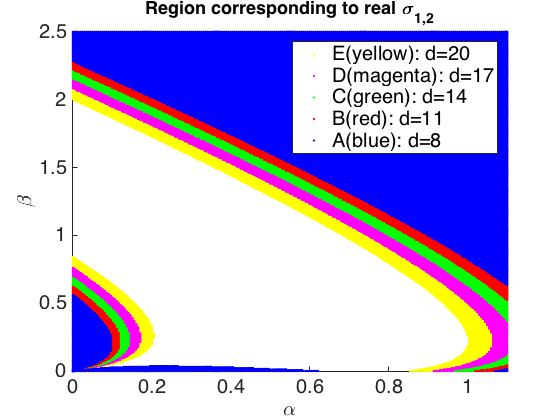}}
		\subfigure[Shift of regions corresponding to real $\sigma_{1,2}$ and subject to condition (\ref{asymp2}) of Theorem \ref{theorem3}]{\includegraphics[width=0.49\textwidth]{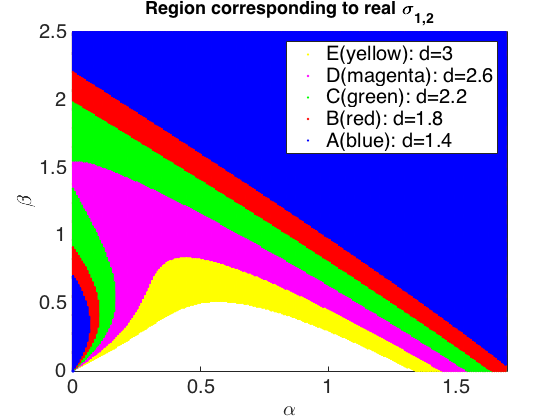}}
	\end{center}
	\caption{The shift in parameter spaces corresponding to real $\sigma_{1,2}$ as a consequence of varying $d$}
	\label{realfig}
\end{figure}
The longer partitioning curves presented in Figure \ref{realcomplex}, indicate the choice of $\alpha$ and $\beta$ for which the eigenvalues $\sigma_{1,2}$ is a pair of real repeated negative values, therefore, parameter spaces bounded by these curves corresponds to a pair of distinct negative real values. A choice of $(\alpha,\beta)$ from these regions corresponds to a global spatio-temporally stable behaviour of the dynamic of system (\ref{polarsystem}). Figure \ref{realsfig} shows the shift of these spatio-temporal stable regions on the admissible parameter space. Any choice of $\alpha$ and $\beta$ from these regions will result in the dynamics of system (\ref{polarsystem}) to exhibit global stability in space as well as in time.
\begin{figure}[ht]
	\begin{center}
		\subfigure[Shift of regions corresponding to a \newline pair of negative distinct real $\sigma_{1,2}$ and sub-\newline ject to condition (\ref{asymp3}) of Theorem \ref{theorem2}]{\includegraphics[width=0.49\textwidth]{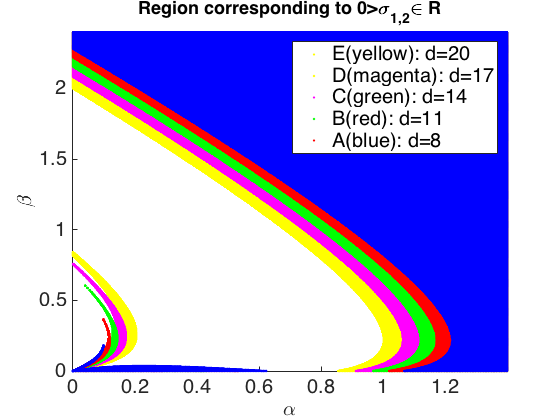}}
		\subfigure[Shift of regions corresponding to a pair of negative distinct real $\sigma_{1,2}$ and subject to condition (\ref{asymp2}) of Theorem \ref{theorem3}]{\includegraphics[width=0.49\textwidth]{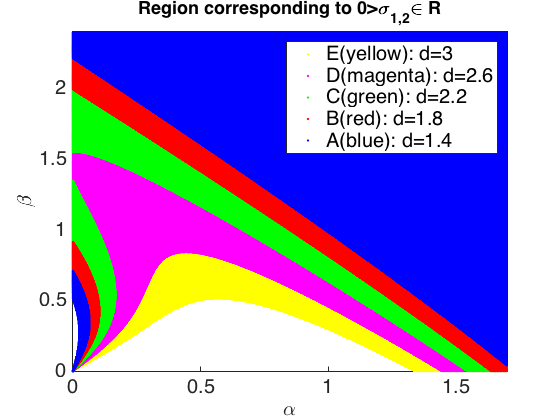}}
	\end{center}
	\caption{The shift in parameter spaces corresponding to negative real distinct $\sigma_{1,2}$ as a consequence of varying $d$}
	\label{realsfig}
\end{figure}
The remaining spaces to analyse are those corresponding to diffusion-driven instability of Turing type under conditions (\ref{asymp2}) and (\ref{asymp3}) on $\rho$. This region corresponds to Turing type instability and under both conditions on $\rho$ these regions exist. It can be noted that near the origin of the admissible parameter space in Figure \ref{realcomplex}, for each value of $d$ the small curves starting at the origin $(\alpha,\beta)=(0,0)$ and curving back to intercept the $\beta$ axis, are the curves on which the eigenvalues $\sigma_{1,2}$ are repeated positive real roots, therefore these curves correspond to diffusion-driven instability of Turing type.
We know that the diffusion-driven instability can also happen, when either $\sigma_1$ or $\sigma_2$ are positive real. Figure \ref{realufig} shows the shift of those regions corresponding to Turing type instability and it can be observed that as $d$ increases, the region in the parameter space enlarges. In Figures \ref{complexs} and \ref{realufig}, all the points specific to a certain colour on the parameter plane are denoted by an alphabetic letter. This is for the purpose to be able to cross reference using set notation to a specific region when summarising the results in Table \ref{table1}.
\begin{figure}[ht]
	\begin{center}
		\subfigure[Shift of regions corresponding to Tur-\newline ing instability where at least one eigenv-\newline alue $\sigma_{1,2}$ is real positive and subject to \newline condition (\ref{asymp3}) of Theorem \ref{theorem2}]{\includegraphics[width=0.49\textwidth]{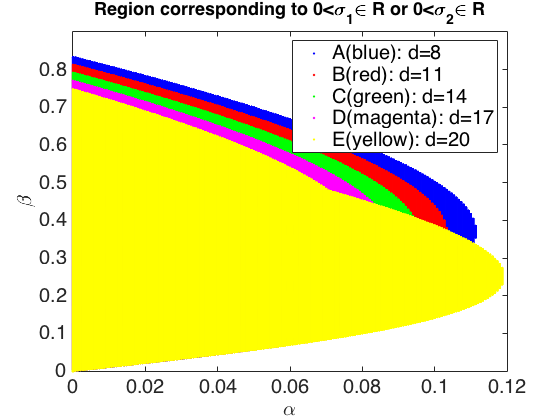}}
		\subfigure[Shift of regions corresponding to Turing instability where at least one eigenvalue $\sigma_{1,2}$ is real positive and subject to condition (\ref{asymp2}) of Theorem \ref{theorem3}]{\includegraphics[width=0.49\textwidth]{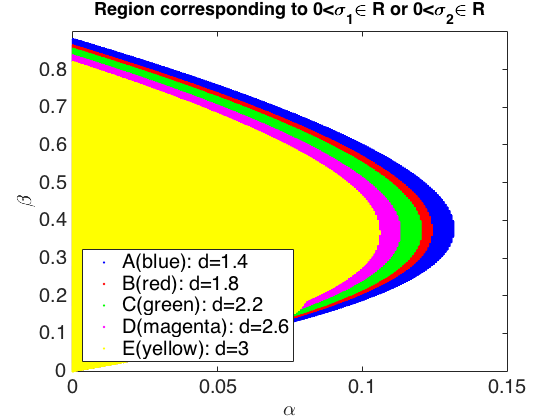}}
	\end{center}
	\caption{The shift in parameter spaces corresponding to at least one positive real eigenvalue $\sigma_{1,2}$ as a consequence of varying $d$}
	\label{realufig}
\end{figure}
\begin{table}
	\caption{The summary of the full classification of parameter spaces for system (\ref{polarsystem}) is presented with the associated numerical verification of the predictions made by Theorems \ref{theorem2} and \ref{theorem3}. Furthermore, it is shown how a certain type of bifurcation space shifts according to the variation of parameter $d$.}
	\centering
	\small
	\tabcolsep=0.3cm
	\noindent\adjustbox{max width=\textwidth}{
		\begin{tabular}{|c|c|c|c|c|c|c|c|}\cline{3-8}\hline
			\multicolumn{3}{|c|}{Stability of USS $(u_s,v_s)$}&\multicolumn{2}{|c|}{Stable regions}&\multicolumn{3}{|c|}{Unstable regions}\\\hline
			\multicolumn{3}{|c|}{Types of USS $(u_s,v_s)$}&Node&Spiral&Turing-instability&Hopf bifurcation& Transcritical bifurcation\\\cline{1-8}
			\multicolumn{3}{|c|}{Figure index}&Figure \ref{realsfig} (a)&Figure \ref{complexs} (a)&Figure \ref{realufig} (a)&Figure \ref{comufig} (a)&Figure \ref{comufig} (b)\\\hline
			\multirow{5}{*}{\begin{turn}{-90}Theorem \ref{theorem2} \end{turn}}&\multirow{5}{*}{\begin{turn}{-90} $\rho$ satisfying \begin{turn}{90}(\ref{asymp3})\end{turn} \end{turn} }&\diaghead{\theadfont{\normalsize} Type of (SS) }{$(d,\gamma,\rho,l)$}{$\sigma_{1,2}$}&$0>\sigma_{1,2}\in\mathbb{R}$&$\sigma_{1,2}\in\mathbb{C}\backslash\mathbb{R}, \text{Re}(\sigma)<0$&$0<\sigma_{1}\in\mathbb{R}$ or $0<\sigma_{2}\in\mathbb{R}$&$\sigma_{1,2}\in\mathbb{C}\backslash\mathbb{R}\text{, Re}(\sigma_{1,2})>0$&$\sigma_{1,2}\in\mathbb{C}\backslash\mathbb{R},\text{ Re}(\sigma_{1,2})=0$\\\cline{3-8}
			&&$(8.0,21,\frac{1}{2},0.27)$&$A$&$A \cup B \cup C \cup D \cup E $&$A$&$A \cup B \cup C \cup D \cup E$& curve $A$\\\cline{3-8}
			&&$(11,21,\frac{1}{2},0.27)$&$A \cup B$&$B \cup C \cup D \cup E$&$A \cup B$&$ B \cup C \cup D \cup E$& curve $B$\\\cline{3-8}
			&&$(14,21,\frac{1}{2},0.27)$&$A \cup B \cup C $&$C \cup D \cup E$&$A \cup B \cup C $&$ C \cup D \cup E$& curve $C$\\\cline{3-8}
			&&$(17,21,\frac{1}{2},0.27)$&$A \cup B \cup C \cup D$&$D \cup E$&$ A \cup B  \cup C \cup D $&$D \cup E$&curve $D$\\\cline{3-8}
			&&$(20,21,\frac{1}{2},0.27)$&$A \cup B \cup C \cup D \cup E$&$E$&$A \cup B \cup C \cup D \cup E$&$E$&curve $E$\\\hline
			\multicolumn{3}{|c|}{Figure index}&Figure \ref{realsfig} (b)&Figure \ref{complexs} (b)&Figure \ref{realufig} (b)&Figure \ref{comufig} (b)&Figure \ref{complexs} (b)\\\hline
			\multirow{5}{*}{\begin{turn}{-90}Theorem \ref{theorem3} \end{turn}}&\multirow{5}{*}{\begin{turn}{-90} $\rho$ satisfying \begin{turn}{90}(\ref{asymp2})\end{turn} \end{turn} }&\diaghead{\theadfont{\normalsize} Type of (SS) }{$(d,\gamma,\rho,l)$}{$\sigma_{1,2}$}&$0>\sigma_{1,2}\in\mathbb{R}$&$\sigma_{1,2}\in\mathbb{C}\backslash\mathbb{R}, \text{Re}(\sigma)<0$&$0<\sigma_{1}\in\mathbb{R}$ or $0<\sigma_{2}\in\mathbb{R}$&$\sigma_{1,2}\in\mathbb{C}\backslash\mathbb{R}\text{, Re}(\sigma_{1,2})>0$&$\sigma_{1,2}\in\mathbb{C}\backslash\mathbb{R},\text{ Re}(\sigma_{1,2})=0$\\\cline{3-8}
			&&$(1.4,1,\frac{1}{2},0.27)$&$E$&$A \cup B \cup C \cup D \cup E$&$A$&$\emptyset$&$\emptyset$\\\cline{3-8}
			&&$(1.8,1,\frac{1}{2},0.27)$&$E \cup D$&$A \cup B \cup C \cup D$&$A \cup B$&$\emptyset$&$\emptyset$\\\cline{3-8}
			&&$(2.2,1,\frac{1}{2},0.27)$&$E \cup D \cup C$&$A \cup B \cup C$&$A \cup B \cup C$&$\emptyset$&$\emptyset$\\\cline{3-8}
			&&$(2.6,1,\frac{1}{2},0.27)$&$E \cup D \cup C \cup B$&$A \cup B$&$A \cup B \cup C \cup D$&$\emptyset$&$\emptyset$\\\cline{3-8}
			&&$(3.0,1,\frac{1}{2},0.27)$&$E \cup D \cup C \cup B \cup A$&$A$&$A \cup B \cup C \cup D \cup E$&$\emptyset$&$\emptyset$\\\hline
	\end{tabular}}
	\label{table1}
\end{table}
\section{Finite element solutions of RDS}\label{fem}
For numerical verifications of the proposed classification of the parameter spaces in particular to visualise the influence of domain size conditions given by Theorems \ref{theorem2} and \ref{theorem3}, the reaction-diffusion system (\ref{polarsystem}) is simulated using the finite element method \cite{paper6, paper12, paper15, paper26, paper27, paper41, book2, book4, book5} on $\Omega$ that consists of annular region with inner radius $a=\frac{1}{2}$ and outer radius $b=1$. This leads to $\rho=b-a=\frac{1}{2}$, which is kept constant throughout the finite element simulations in this section. Conditions (\ref{asymp3}) and (\ref{asymp2}) are in turn satisfied by varying the values of $d$ and $\gamma$ accordingly. The advantage of such strategy is that it reduces the computational cost significantly in the sense that an efficient degree of freedom on the finite element triangulation can be consistently used to obtain results under both of the proposed conditions in Theorems \ref{theorem2} and \ref{theorem3} respectively. Therefore, to avoid such unnecessary computational cost by varying $\rho$ we proceed with a fixed $\rho=\frac{1}{2}$ and vary the constants $d$ and $\gamma$ to meet the relevant restriction required on the size of $\rho$. Due to the curved boundary of $\Omega$, the triangulation is obtained through an application of an iterative algorithm using a technique called \textit{distmesh} \cite{paper40, paper41}. 
The algorithm for \textit{distmesh} was originally developed in MATLAB by Persson and Strang \cite{paper40, paper41} for generating uniform and non-uniform refined meshes on two and three dimensional geometries. \textit{Distmesh} utilises signed-distance function $d(x,y)$, which is negative inside the discretised domain $\Omega$ and is positive outside $\partial\Omega$. The construction of distmesh triangulation is an iterative process using a set of two interactive algorithms, one of which controls the displacement of nodes within the domain and the other ensures that the consequences of node displacement does not violate the properties of the Delaunay triangulation \cite{paper43}. For details on how to generate meshes using \textit{distmesh} we refer the interested reader to the joint work by Persson and Strang presented in \cite{paper40, paper41, paper44}.
An annular region $\Omega$ of thickness $\rho=\frac{1}{2}$ is descritised through the application of \textit{distmesh} \cite{paper40} to generate a triangular mesh for simulating the finite element solution of (\ref{polarsystem}). The two circles forming the inner and out boundaries of $\Omega$ are concentrically centred at the origin of cartesian plane. The annular region is discretised by $6340$ triangles consisting of $3333$ vertices. In each simulations in this section the initial conditions are continuous bounded functions with pure spatial dependence as small perturbations in the neighbourhood of the uniform steady state $(u_s,v_s)=(\alpha+\beta, \frac{\beta}{(\alpha+\beta)^2})$ \cite{book1, book7, paper38, paper17} in the form
\begin{equation}
\begin{cases}
u_0(x,y)=\alpha+\beta+0.0016\cos(2\pi(x+y))+0.01\sum_{i=1}^8\cos(i\pi x),\\
v_0(x,y)= \frac{\beta}{(\alpha+\beta)^2}+0.0016\cos(2\pi(x+y))+0.01\sum_{i=1}^8\cos(i\pi x).
\end{cases}
\label{initial}
\end{equation} 
The values of parameters $\alpha$ and $\beta$ are verified from all of those regions that correspond to some type of diffusion-driven instability. The numerical values of the parameters corresponding to each of the simulations in this section are presented in Table \ref{Table2}.    
Figure \ref{unstable1} (a) shows the evolution of a spatial pattern as a consequence of choosing $(\alpha,\beta)$ from Turing region under condition (\ref{asymp2}) indicated in Figure \ref{realufig} (b). 
Depending on the initial conditions and the mode of the eigenfunctions, the spatially periodic pattern provided by parameter spaces in Figure \ref{realufig} (b) is expected to be a combination of radial and angular stripes or spots. 
Once the initial pattern is formed by the evolution of the dynamics, then the system is expected to uniformly converge to a Turing-type steady state, which means the initially evolved spatial pattern becomes temporally invariant as time grows. 
The simulation of Figure \ref{unstable1} was executed for long enough time such that the discrete time derivative of solutions $u$ and $v$ decaying to a threshold of $5\times 10^{-4}$ in the discrete $L_2$ norm. Figure \ref{unstable1} (b) demonstrates the behaviour of the discrete time derivatives of both species for the entire period of simulation time until the threshold was reached. 
It is observed that after the initial Turing-type instability, the evolution of the system uniformly converges to a spatially patterned steady state.   
\begin{figure}[ht]
	\begin{center}
		\subfigure[Pattern evolution is res-\newline tricted to spatial periodici-\newline ty for $\alpha$ and $\beta$ in the Tur-\newline ing space under condition \newline  (\ref{asymp2}) on the domain size p-\newline arameter $\rho$]{\includegraphics[width=0.32\textwidth]{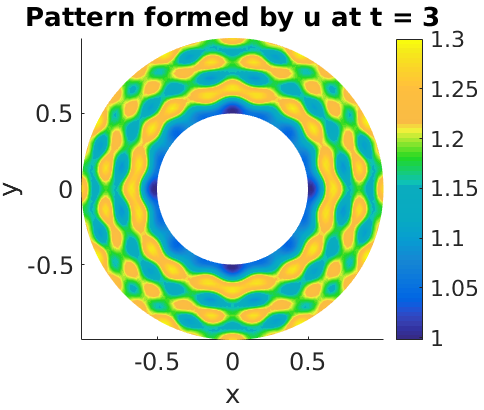}}
		\subfigure[Pattern evolution is res-\newline tricted to spatial periodici-\newline ty  for $\alpha$ and $\beta$ in the Tur-\newline ing space under condition \newline (\ref{asymp2}) on the domain size pa-\newline rameter $\rho$]{\includegraphics[width=0.32\textwidth]{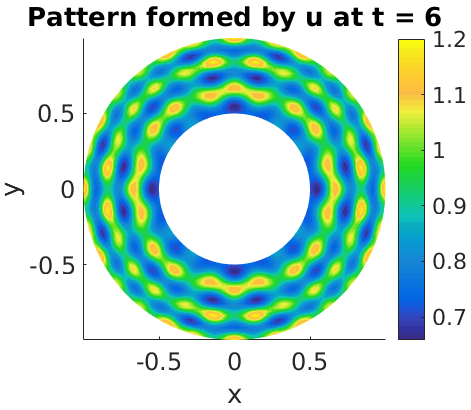}}
		\subfigure[Instability shown through the evolution of the discrete $L_2$ norm of the discrete time derivatives of the solutions $u$ and $v$\newline\newline]{\includegraphics[width=0.32\textwidth]{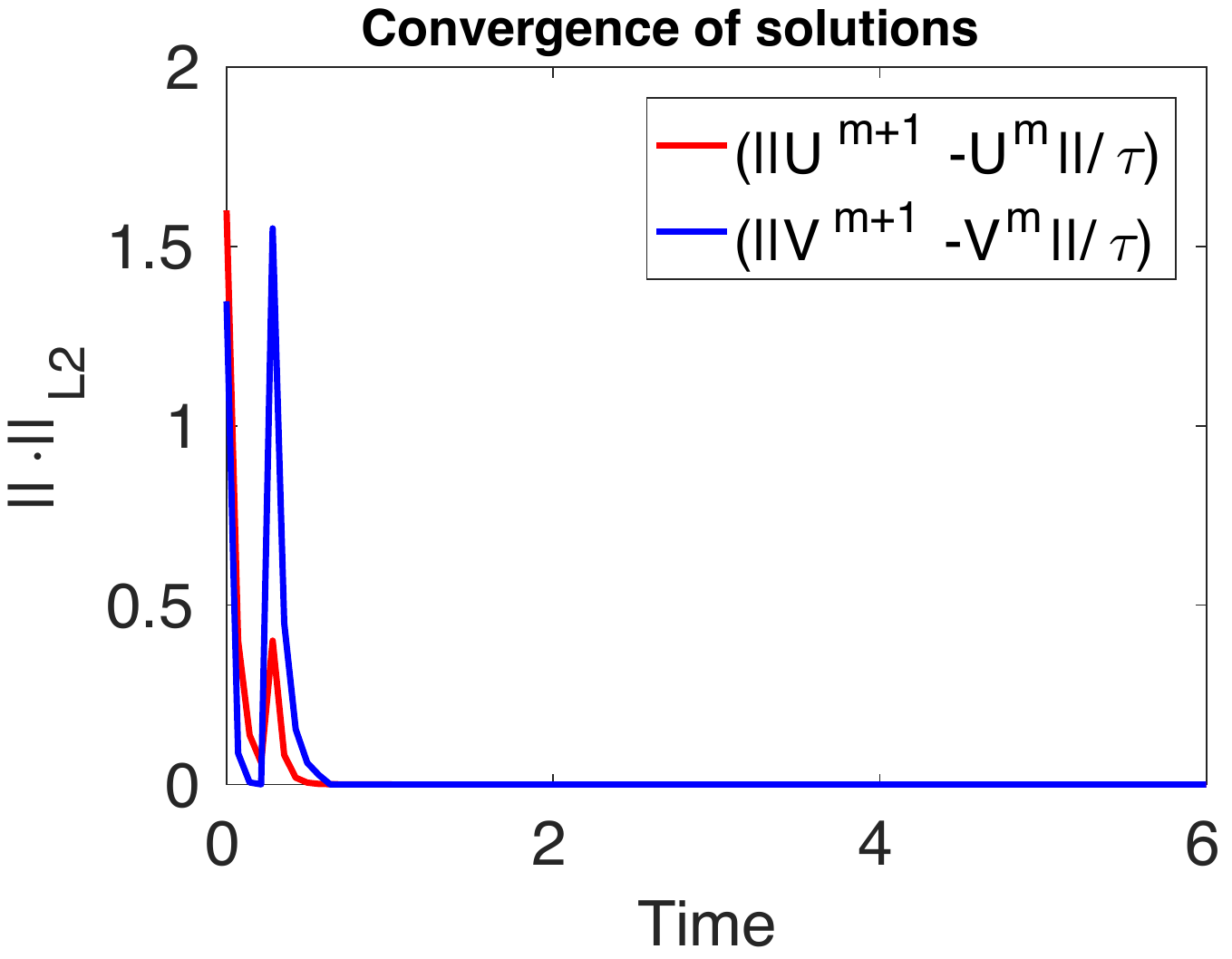}}
	\end{center}
	\caption{When $\rho$ is bounded by a combination of $d$ and $\gamma$ (as shown in Table \ref{Table2}) according to condition (\ref{asymp2}), then the only admissible pattern is a spatially periodic pattern for $(\alpha,\beta)$ from the Turing-space shown in Figure \ref{realufig} (b)}
	\label{unstable1}
\end{figure}
Figure \ref{unstable2} presents two snapshots to show how the spatially periodic pattern is evolved to a Turing type steady state, when parameters $\alpha$ and $\beta$ are chosen from Turing region and $\rho$ satisfying condition (\ref{asymp3}), with respect to $d$ and $\gamma$. For simulations in Figure \ref{unstable2}, parameters $\alpha$ and $\beta$ are chosen from regions presented in Figure \ref{realufig} (a), the dynamics within these regions evolve to a spatial pattern with global temporal stability.  
The remaining two unstable regions in the admissible parameter space presented in Figure \ref{comufig} correspond to spatio-temporal periodicity.
\begin{figure}[ht]
	\begin{center}
		\subfigure[The stage of evolving s-\newline patially periodic pattern a-\newline t $t=3$ when, $\alpha$ and $\beta$ are -\newline chosen from Turing space u-\newline nder condition (\ref{asymp3}) on the r-\newline adius $\rho$]{\includegraphics[width=0.32\textwidth]{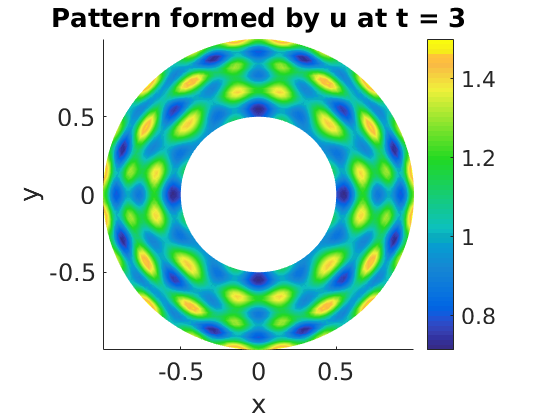}}
		\subfigure[Spatially periodic patt-\newline ern at $t=6$ is as expect-\newline ed converging to the Tur-\newline ing type steady state with-\newline out allowing the initial p-\newline attern to be deformed]{\includegraphics[width=0.32\textwidth]{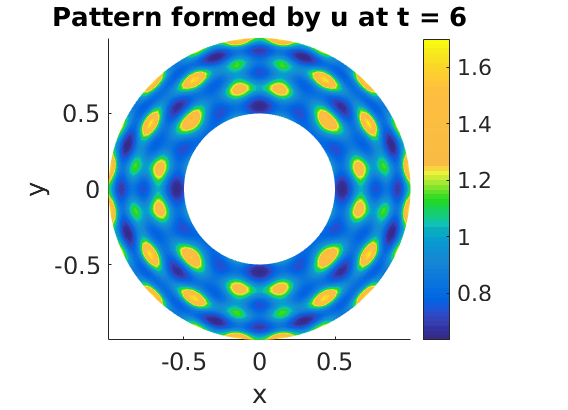}}
		\subfigure[Instability and convergence is shown through the behaviour of the discrete time derivatives of the solutions $u$ and $v$]{\includegraphics[width=0.32\textwidth]{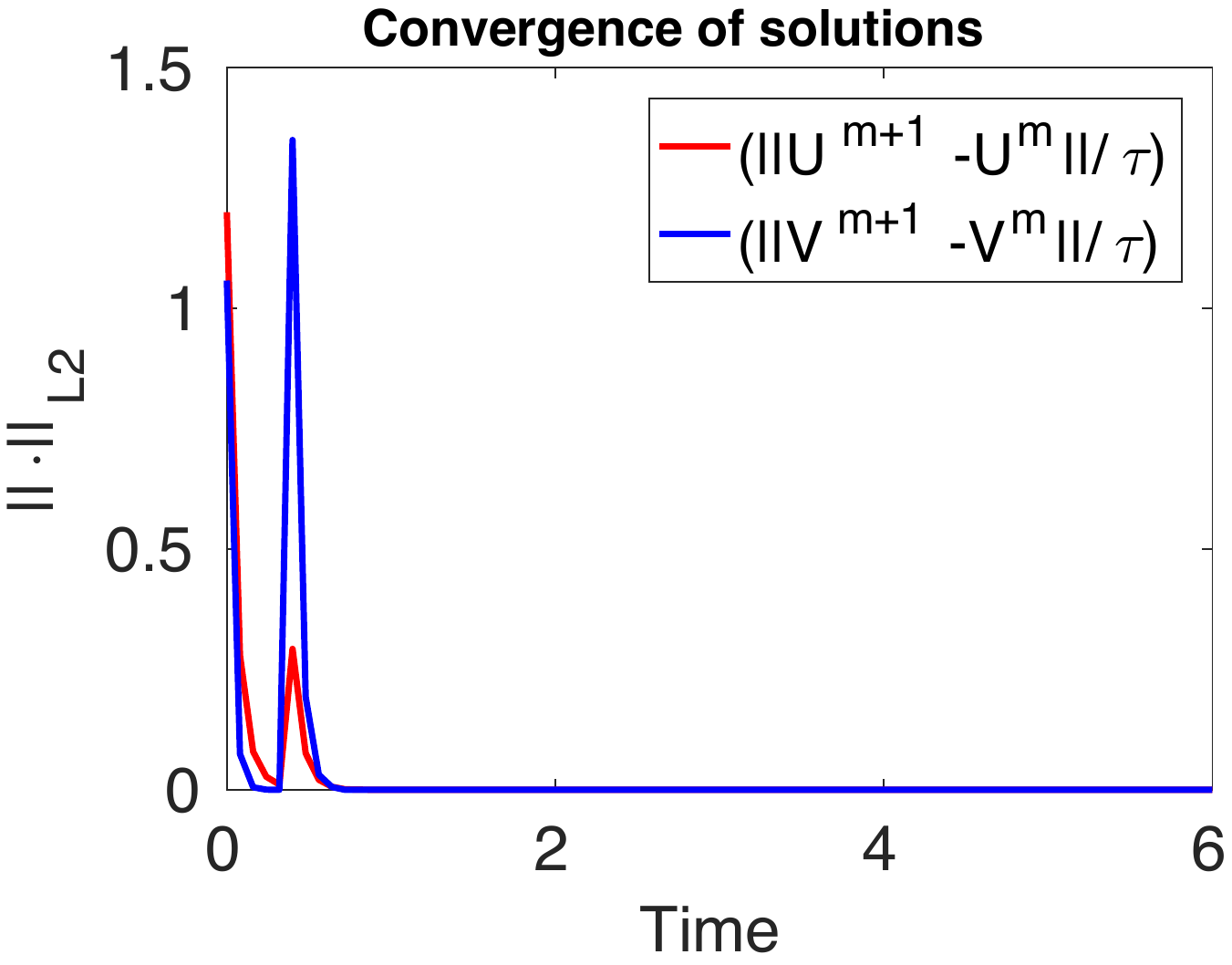}}
	\end{center}
	\caption{When $\rho$ is large with respect to the combination of $d$ and $\gamma$ (as shown in Table \ref{Table2}) according to condition (\ref{asymp3}), then the dynamics admit spatial diffusion-driven instability for $(\alpha,\beta)$ from the Turing-space}
	\label{unstable2}
\end{figure}
Choosing parameters from regions in Figure \ref{comufig} (a) admits temporal periodic behaviour in the dynamics of system (\ref{polarsystem}) as shown in Figure \ref{hopf}. 
It is worth noting that the temporal period between the successive transitional temporal instabilities from one type of spatial pattern to another grows larger with time. The initial pattern is obtained at around $t\approx1$, which becomes unstable during the transition to the second temporal period at $t\approx5$ in Figure \ref{hopf} (b). At $t\approx8$ the system undergoes a third period of instability and reaches a different spatial pattern at $t\approx12$ shown in Figure \ref{hopf} (c).
The fourth period of temporal instability is reached at $t\approx 20$, which converges to the fourth temporally-local but spatially periodic steady state at $t\approx28$ presented in Figure \ref{hopf} (d). It follows that when parameters are chosen from the Hopf bifurcation region then the temporal gaps in the dynamics of system (\ref{polarsystem}) between successive transitional instabilities from one spatial pattern to another is approximately doubled as time grows. It is speculated that the temporal period-doubling behaviour is connected to the analogy of unstable spiral behaviour in the theory of ordinary differential equations \cite{book12}.
\begin{figure}[ht]
	\begin{center}
		\subfigure[First temporal period\newline  evolving the initial spatia-\newline lly periodic pattern at \newline $t=0.5$]{\includegraphics[width=0.32\textwidth]{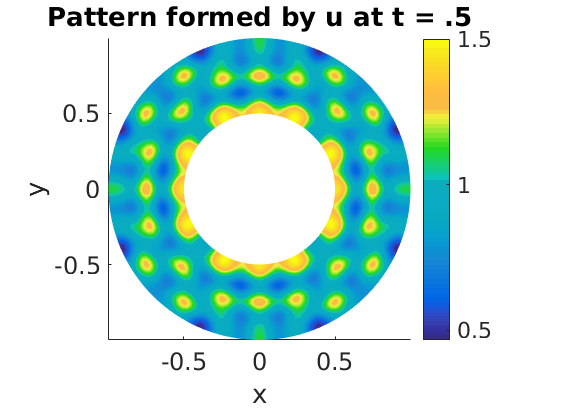}}
		\subfigure[Spatial pattern evolv-\newline ed after the first transi-\newline tional instability and duri-\newline ng the second temporal\newline period at $t=3$]{\includegraphics[width=0.32\textwidth]{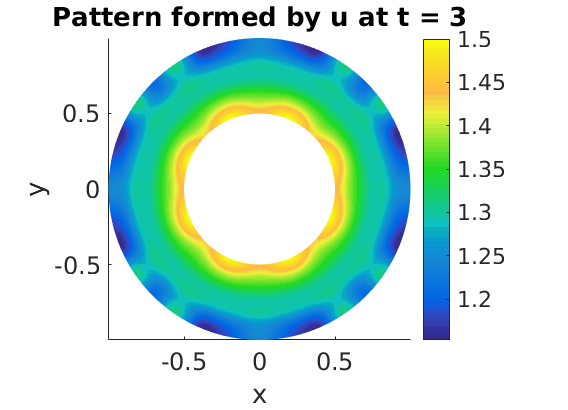}}
		\subfigure[Spatial pattern evolved after  the second transition of temporal instability and during the third temporal period at $t=9$]{\includegraphics[width=0.32\textwidth]{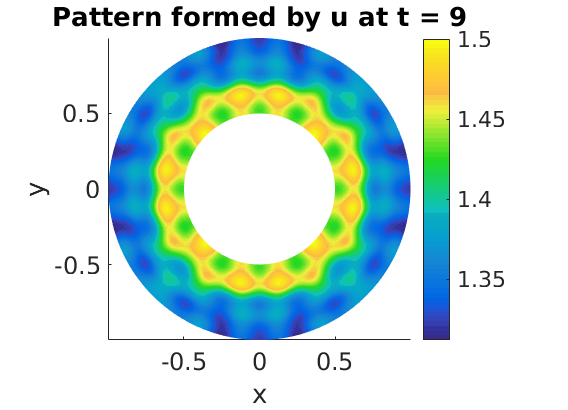}}
		\subfigure[Spatial pattern after \newline the third transition of tem-\newline poral instability and duri-\newline ng the fourth temporal per-\newline iod obtained at $t=15$]{\includegraphics[width=0.32\textwidth]{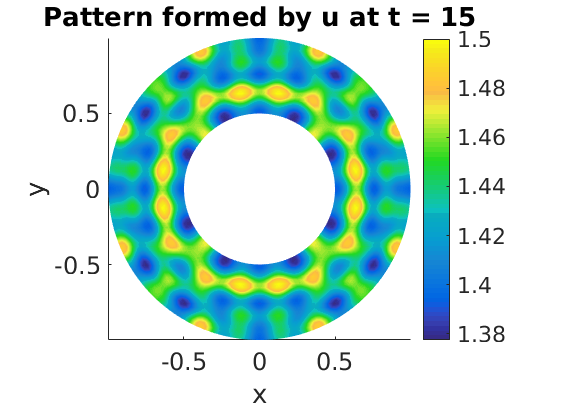}}
		\subfigure[Spatial pattern after the\newline  fourth transition of tempo-\newline ral instability and during-\newline the  fifth temporal period o-\newline btained at $t=30$]{\includegraphics[width=0.32\textwidth]{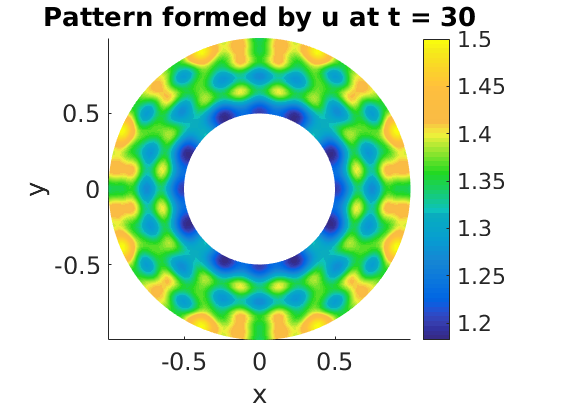}}
		\subfigure[Spatio-temporal periodicity in the dynamics measured in dicrete $L_2$ norm of the successive discrete time derivative of the solutions $u$ and $v$]{\includegraphics[width=0.32\textwidth]{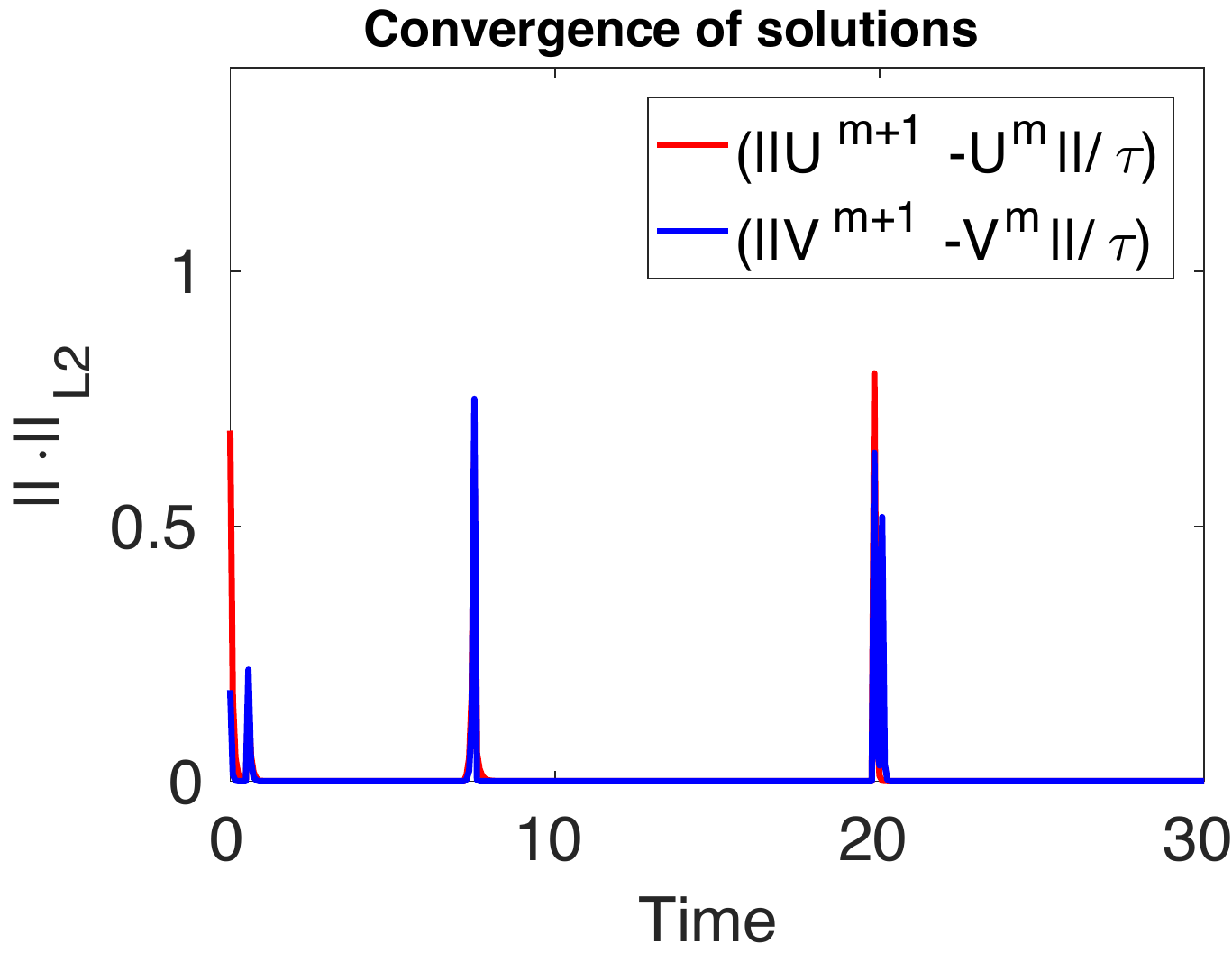}}
	\end{center}
	\caption{When $\rho$ is large with respect to the combination of $d$ and $\gamma$ (as shown in Table \ref{Table2}) according to condition (\ref{asymp3}), then the dynamics admit spatio-temporal periodic behaviour for $(\alpha,\beta)$ from Hopf bifurcation region presented in Figure \ref{comufig} (a)}
	\label{hopf}
\end{figure}
Theorem \ref{theorem2} states that, if the thickness $\rho$ of the annular region $\Omega$ is chosen according to condition (\ref{asymp3}) with respect to $d$ and $\gamma$, and the parameters $(\alpha,\beta)$ are chosen from the curves corresponding to transcritical bifurcation indicated in Figure \ref{comufig} (b), then the dynamics of system (\ref{polarsystem}) are expected to exhibit spatio-temporal periodic behaviour. This kind of behaviour in the dynamics is referred to as the limit cycles \cite{paper51, book12}. Figure \ref{unstable3}  shows this spatio-temporal periodic behaviour in the evolution of the numerical solution of system (\ref{polarsystem}). This is the case corresponding to parameters that ensure the eigenvalues to be purely imaginary, therefore, it can be observed that the temporal instability occurs unlike the Hopf bifurcation, with constant periods along the time axis, which verifies the theoretical prediction of the transcritical bifurcation.
\begin{figure}[ht]
	\begin{center}
		\subfigure[Initial spatially periodic pattern (angul-\newline ar stripes) at $t=0.2$ when, $\alpha$ and $\beta$are cho-\newline sen from the region of the transcritical bif-\newline urcation under condition (\ref{asymp3}) on the thi-\newline ckness $\rho$]{\includegraphics[width=0.49\textwidth]{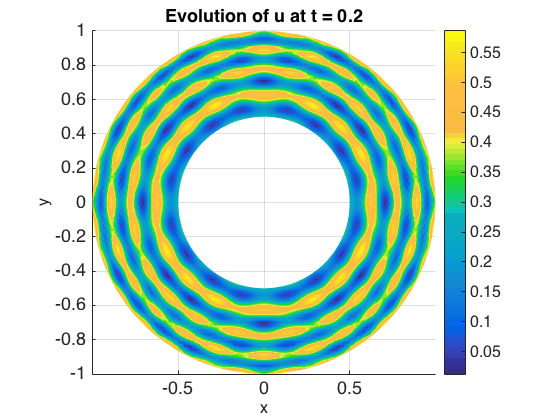}}
		\subfigure[Spatially periodic pattern evolves to change topology to(spots) at $t=3$ for the same choice of parameters as in Figure (a)]{\includegraphics[width=0.49\textwidth]{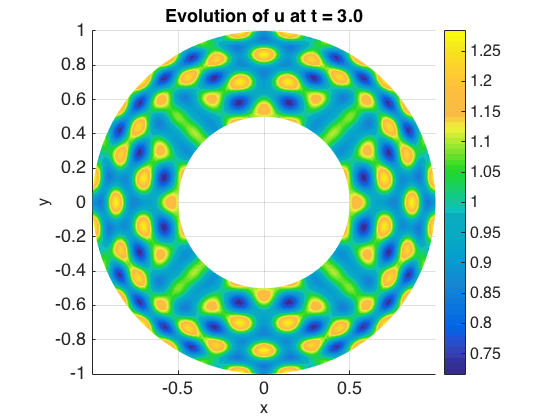}}
		\subfigure[At $t=6$ the spots undergo another \newline period of temporal instability to emerge \newline angular stripes as the initial pattern]{\includegraphics[width=0.49\textwidth]{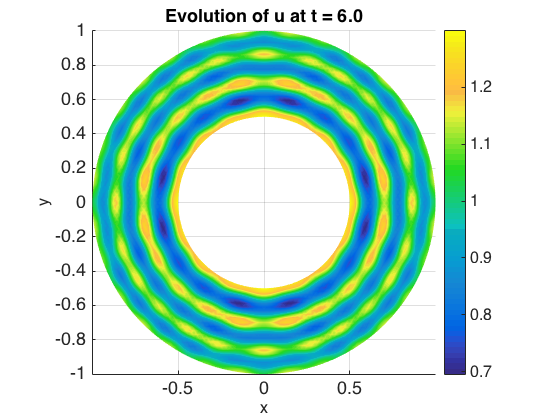}}
		\subfigure[At $t=9$ the pattern of the second temporal period emerges again, indicating that temporally the dynamics behave in an alternating way between the spots and angular stripes]{\includegraphics[width=0.49\textwidth]{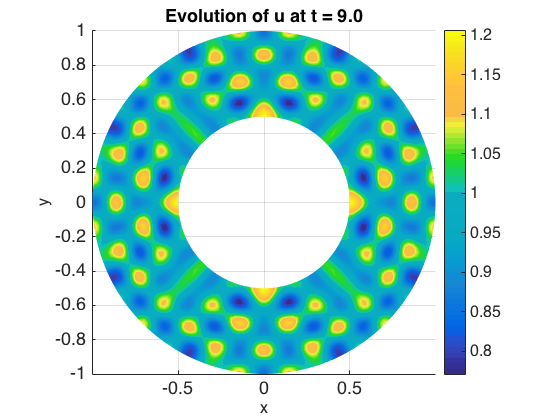}}
	\end{center}
	\caption{When $\rho$ is large with respect to the combination of $d$ and $\gamma$ (as shown in Table \ref{Table2}) according to condition (\ref{asymp3}), then the dynamics of (\ref{polarsystem}) can also admit spatio-temporal diffusion-driven instability for $(\alpha,\beta)$ from the transcritical birfucation curves indicated in Figure \ref{comufig} (b)}
	\label{unstable3}
\end{figure}
It is worth noting that the transitional instability from angular stripes in Figure \ref{unstable3} (a) to the spots in Figure \ref{unstable3} (b), the discrete $L_2$ norm of the discrete time-derivative of the activator $u$ exceeds, in magnitude compared to, that of the inhibitor $v$. However, during the second temporal period when the spots in Figure \ref{unstable3} (b) turn back into angular stripes in Figure \ref{unstable3} (c), then the discrete $L_2$ norm of the time-derivative of the inhibitor $v$ exceeds in magnitude than that of the activator $u$. This alternating behaviour can be clearly observed in Figure \ref{unstable3} (e), where in the annotated legend $U$ and $V$ denote the discrete solutions of the activator $u$ and that of the inhibitor $v$.  It can further be understood from Figure \ref{unstable3} (e), that if $(\alpha,\beta)$ are chosen from the curves of the transcritical bifurcation given in Figure \ref{comufig} (b), then the frequency of temporal periods is predicted to remain constant for all times, resulting in a constant interchanging behaviour between different spatial patterns. 
\begin{table}
	\caption{The choice of parameters $(\alpha,\beta)$ for each simulation and the choice of $(l,d,\gamma)$ subject to the relevant condition referred to in third row. Each simulation is run with time-step of $1\times10^{-3}$.} 
	\centering
	\small
	\tabcolsep=0.3cm
	\noindent\adjustbox{max width=\textwidth}{
		\begin{tabular}{|c |c |c |c| c|}
			\cline{1-1}
			\hline
			Plot index&  Figure \ref{unstable1} & Figure \ref{unstable2} &  Figure \ref{hopf} & Figure \ref{unstable3}\\
			\hline
			\diaghead{\theadfont{\normalsize} Type of (SSSS) }{Parameters}{Instability}&\shortstack{Turing type instability\\Spatial pattern}& \shortstack{Turing type instability\\Spatial pattern}&\shortstack{Hopf bifurcation\\Spatial and temporal pattern} &\shortstack{Transcritical bifurcation\\Spatial and temporal pattern} \\
			\hline
			$(\alpha,\beta)$ & $(0.09, 0.45)$ & $(0.08, 0.40)$ & $(0.05, 0.55)$ & $(0.05, 0.625)$ \\
			\hline
			$(l,d, \gamma)$ & $(1.3,10, 250)$& $ (1.3,5, 630)$ & $ (1.3,5,730)$ & $(1.3,5,730)$\\
			\hline
			Condition on $\Omega$ & (\ref{asymp2}) &  (\ref{asymp3}) &(\ref{asymp3}) & (\ref{asymp3})\\
			\hline
			Simulation time &  $6$ &  $6$ & $30$ & $9$\\
			\hline
			CPU time (sec) & $648.21$ &  $631.63$ & $4230.55$ & $1097.32$\\
			\hline
	\end{tabular}}
	\label{Table2}
\end{table}
\section{Conclusion}\label{conclusion}
A reaction-diffusion model of \textit{activator-depleted} class is analysed on polar coordinates on a closed bounded but non-compact domain that consists of an annular region. The model is subject to homogeneous boundary conditions of Neumann type. From the linearisation, the spectrum of Laplace operator on the annular domain was analysed and closed form power series solutions are obtained for the relevant eigenvalue problem. It is further shown that  the spectrum of the diffusion operator in question consists of semi-discrete real positive eigenvalues corresponding to an infinite set of complex valued eigen-functions. The solutions of the relevant eigenvalue problem are simulated using the spectral method on a two dimensional annular geometry, which was conducted through the application of a particular depicting technique, since the solutions are complex valued.

\begin{wrapfigure}{l}{0.45\textwidth}
	\begin{center}
		\includegraphics[width=0.3\textwidth]{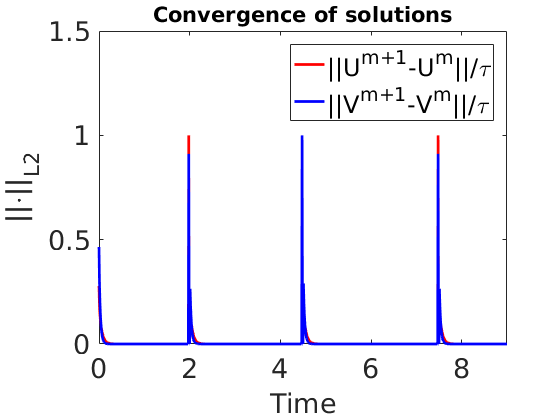}
	\end{center}
	\caption{Plot of the discrete $L_2$ norm of the  time-derivative of $u$ and $v$ showing the spatio-temporal behaviour of the solutions for successive time-steps}
	\label{Fig1}
\end{wrapfigure} 

The eigen-functions and the corresponding eigenvalues are incorporated through the application of linear stability theory to derive analytical results relating the thickness of the annular domain to the reaction and diffusion parameters. It is found that a lower bound exists on the thickness of a two dimensional annular domain with respect to the reaction and diffusion parameters that can admit the three types of bifurcations namely Turing, Hopf and the transcritical bifurcation. Furthermore, it is found that an upper bound also exists on the thickness of the two dimensional annular domain that forbids temporal periodicity such as Hopf and transcritical bifurcations, and still admitting a Turing type behaviour in the evolution of the dynamics. Finite element simulations were used to verify the analytically proven results through numerical solutions of \textit{activator-depleted} class reaction-diffusion model on an annular region. 
A biological application of the current study is to model the reaction and diffusion of chemo-taxis with the immune cells of a tumour after reaching the hypoxic stage. Cancer tumours evolve to grow in size by attracting local capillaries for oxygen and nutrition. As the process continues and the tumour grows larger in size. It reaches a point when the attracting capillaries can no longer supply nutrients to the tumour due to its size and the extensive consumption of nutrition contributing to its growth on the surface. When the interior of the tumour lacks oxygen and nutrition, the tumour enters a stage called \textit{hypoxia}, during which the cells in the interior of the tumour starve to death and with long enough time, the activities of uncontrolled growth concentrates on a hollow sphere. At this stage of the tumour the hollow sphere can be modelled by a two dimensional shell using rotational symmetry with respect to the zenith angle, which leads to a two dimensional annular region. 
The current work can also be extended to the analysis of bulk-surface reaction-diffusion systems. The evolution of pattern in a coupled bulk-surface reaction-diffusion system can be investigated to find the influence of the area of the surface and the volume of the bulk on the evolution of dynamics.
A further extension of this work is to apply the framework of this study to reaction-diffusion systems on evolving domains. Conditions that are found in the scope of this study can further be explored, whether, with the growth of the domain, these conditions still hold to influence the evolution of pattern formation, or is there a threshold for the domain-size, beyond which these conditions no longer hold?
The framework of the current work can also be applied to \textit{activator-depleted} reaction-diffusion model with linear cross diffusion to find whether these conditions on the domain size hold in the presence of cross diffusion. The parameter spaces are expected to significantly change as the diffusion matrix is no longer diagonal in the presence of linear cross diffusion.

\section*{Acknowledgement}
WS acknowledges support of the School of Mathematical and Physical Sciences Doctoral Training studentship. AM acknowledges support from the Leverhulme Trust Research Project Grant (RPG-2014-149) and the European Union's Horizon 2020 research and innovation programme under the Marie Sklodowska-Curie grant agreement No 642866. AM's work was partially supported by the Engineering and Physical Sciences Research Council, UK grant (EP/J016780/1). The authors (WS, AM) thank the Isaac Newton Institute for Mathematical Sciences for its hospitality during the programme (Coupling Geometric PDEs with Physics for Cell Morphology, Motility and Pattern Formation; EPSRC EP/K032208/1).  AM was partially supported by a fellowship from the Simons Foundation. AM is a Royal Society Wolfson Research Merit Award Holder, generously supported by the Wolfson Trust. 

\section*{Conflict of interest}
The authors declare no conflict of interest.

\end{document}